\documentclass[11pt]{article}
\usepackage{latexsym,a4wide}
\usepackage{amssymb}
\usepackage{amsmath}
\usepackage{color}
\usepackage{graphicx}
\usepackage{amsthm}
\allowdisplaybreaks

\newcommand \bel {\begin{equation}\label}
\newcommand \ee {\end{equation}}
\newcommand \be {\begin{equation}}

\newcommand \RR {\mathbb R}
\newcommand \HH {\mathbb H}

\newcommand \CC {\mathbb C}
\newcommand \LL {\mathbb L}
\newcommand \NN {\mathbb N}
\newcommand \TT {\mathbb T}

\newcommand \del \partial

\newcommand \Bcal {\mathcal B}

\newcommand \Ccal {\mathcal C}

\newcommand \Ocal {\mathcal O}
\newcommand \Lcal {\mathcal L}
\newcommand \Jcal {\mathcal J}

\newcommand \bei {\begin{itemize}}
\newcommand \eei {\end{itemize}}

\def \eps {\varepsilon}

\newtheorem{theorem}{\color{black}\indent Theorem}[section]
\newtheorem{lemma}{\color{black}\indent Lemma}[section]
\newtheorem{proposition}{\color{black}\indent Proposition}[section]
\newtheorem{definition}{\color{black}\indent Definition}[section]
\newtheorem{remark}{\color{black}\indent Remark}[section]

\usepackage{titletoc,tocloft}
\dottedcontents{section}[1.5em]{}{1.3em}{.6em}

\begin{document}
\large
\title{\bf Nonlinear stability of explicit self-similar solutions for the timelike extremal hypersurfaces in $\mathbb{R}^{1+3}$}
\author{
{\sc Weiping Yan}\thanks{School of Mathematics, Xiamen University, Xiamen 361000, P.R. China. Email: yanwp@xmu.edu.cn.}
}
\date{November 20, 2018}

\maketitle
\begin{abstract}  
This paper is devoted to the study of the singularity phenomenon of timelike extremal hypersurfaces in Minkowski spacetime $\mathbb{R}^{1+3}$. We find that there are two explicit lightlike self-similar solutions to a graph representation of timelike extremal hypersurfaces in Minkowski spacetime $\mathbb{R}^{1+3}$, the geometry of them are two spheres.
The linear mode unstable of those lightlike self-similar solutions for the radially symmetric membranes equation is given. 
After that, we show those self-similar solutions of the radially symmetric membranes equation
are nonlinearly stable inside a strictly proper subset of the backward lightcone. This means that the dynamical behavior of those two spheres is as attractors.
Meanwhile, we overcome the double roots case (the theorem of Poincar\'{e} can't be used) in solving the difference equation
by construction of a Newton's polygon when we carry out the analysis of spectrum for the linear operator.

\end{abstract}

\tableofcontents


\section{Introduction and main results} 
\subsection{Introduction}
\setcounter{equation}{0}
The timelike minimal surface equation arises in string theory and geometric minimal surfaces theory in Minkowski space. There has been discovered that the behavior of string theory in spacetimes that develop singularities \cite{W}. Meanwhile,
the study of singularity is one of most important topics in physics and mathematics theory, which corresponds to a physical event. It can also imply that some essential physics is missing from the equation in question, which should thus be supplemented with additional terms. 
Hence it is a nature problem to study the singularity formation of timelike minimal surface equation.

When the spacial dimension is one, the timelike minimal surface equation is so-called Born-Infeld equation (or relavisitive string equation).
Eggers-Hoppe \cite{Hop1,Hop2} first gave some interesting description of self-similar singularity to timelike extremal hypersurfaces, meanwhile, the swallowtail singularity was also been given by the study of the string solution in \cite{Egg}. After that, Nguyen-Tian \cite{tian} proved the existence of blowup solution when the string moving in Einstein vaccum spacetime. One can see the well-posedness theory and related results in \cite{Ba,Kong,Lin,Mi,Yan1} for this kind of equations. 

Let $\mathcal{M}$ be a timelike $(M+1)$-dimensional hypersurface, and $(\mathbb{R}^{D},g)$ be a $D$-dimensional Minkowski space, and $g$ be the Minkowski metric with $g(\partial_t,\partial_t)=1$. At any time $t$, the spacetime volume in $\mathbb{R}^{D}$ of timelike hypersurface $\mathcal{M}$ can be described as a graph over $\mathbb{R}^M$, which satisfies
\begin{equation}\label{E1-0}
\mathcal{S}(u)=\int_{\mathbb{R}}\int_{\mathbb{R}^{M}}\sqrt{1-|\partial_t u|^2+|\nabla u|^2}d^Mxdt.
\end{equation}
Critical points of action integral (\ref{E1-0}) give rise to submanifolds $\mathcal{M}\subset\mathbb{R}^{D}$ with
vanishing mean curvature, i.e. timelike extremal hypersurfaces. The Euler-Lagrange equation of (\ref{E1-0}) is
\begin{equation}\label{ENNN1-1}
\partial_t\left(\frac{\partial_t u}{\sqrt{1-|\partial_t u|^2+|\nabla_xu|^2}}\right)-\sum_{i=1}^M\partial_{x_i}\left(\frac{\partial_{x_i}u}{\sqrt{1-|\partial_t u|^2+|\nabla_xu|^2}}\right)=0.
\end{equation}
Thus finding the solution of (\ref{ENNN1-1}) is equivalent to solve the equation \cite{Hop1,Hop2} 
\bel{E1-1}
(1-u_{\alpha}u^{\alpha})\square_g u+u^{\beta}u^{\alpha}u_{\alpha\beta}=0,
\ee
where $\forall\alpha,\beta=0,1,2,\ldots,M$, $u_{\alpha}=\frac{\partial u}{\partial x^{\alpha}}$, $u_{\alpha\beta}=\frac{\partial^2 u}{\partial x^{\alpha}\partial x^{\beta}}$ and $\square_g u=u_{\alpha\beta}g^{\alpha\beta}$.

Let $M=1+2$, $D=4$ and $r=|x|$, equation (\ref{E1-1}) is reduced into the radially symmetric membranes equation
\bel{E1-2}
u_{tt}-u_{rr}-\frac{u_r}{r}+u_{tt}u_r^2+u_{rr}u_t^2-2u_tu_ru_{tr}+\frac{1}{r}u_ru_t^2-\frac{1}{r}u_r^3=0.
\ee

We supplement equation (\ref{E1-2}) with an initial data
\begin{equation}\label{E1-2R1}
u(0,r)=u_0(r),~~u_{t}(0,r)=u_1(r).
\end{equation}

If $u(t,r)$ is a solution of (\ref{E1-2}), there exists an exact scaling invariance
\begin{equation}\label{E1-3}
u(t,r)\mapsto u_{\lambda}(t,r)=\lambda u(\lambda^{-1}t,\lambda^{-1}r),\quad for \quad any \quad constant \quad \lambda>0,
\end{equation}
and it is a mass conservation dynamics, i.e.
$$
\int_{\RR}\left(\frac{\partial_t u}{\sqrt{1-|\partial_t u|^2+|\nabla_xu|^2}}\right)dx_i~is~conserved~along~the~dynamics.
$$

In general, quasilinear wave equations are energy supercritical, thus the smooth finite energy initial data leads to finite time blowup of solutions, and the blowup rate is like the self-similar blowup rate. Hence, we expect the radially symmetric membranes equation (\ref{E1-2}) admits self-similar blowup solutions. Eggers-Hoppes \cite{Hop1} gave a detail discussion on the existence of self-similar blowup solutions (not explicit self-similar solutions) to the radially symmetric membranes equation (\ref{E1-2}). Meanwhile, they gave some numerical analysis results on the formation of singularity for equation (\ref{E1-1}). 

\subsection{Main result}

In view of the radially symmetric membranes equation (\ref{E1-2}), it is natural to investigate whether explicit singular solutions do exist and whether they are stable.
In the present paper, we first show there are two explicit self-similar blowup solutions to (\ref{E1-2}), then we prove nonlinearly stable of them inside a strictly subset of the backward lightcone. Here two explicit self-similar solutions are lightlike solutions, i.e. which propagate with the speed of light. The existence of it is dimension independent \cite{Hop0}.

\begin{theorem}
\begin{itemize}
\item
The radially symmetric membranes equation (\ref{E1-2}) has two explicit lightlike self-similar solutions
$$
u_T^{\pm}(t,r)=\pm(T-t)\sqrt{1-(\frac{r}{T-t})^2}, \quad\quad (t,r)\in(0,T)\times[0,T-t],
$$
where the positive constant $T$ denotes the maximal existence time. 

Moreover, two explicit lightlike self-similar solutions admit \textbf{smooth initial data and finite energy} in $(0,T)\times(0,T-t]$.

\item
Those explicit lightlike self-similar solutions are nonlinearly stable inside a strictly proper subset of the backward lightcone $\Omega_{T-t}$, i.e.
there exist a positive constant $\sigma\in(0,1)$ and a small positive constant $\eps$ depending on $\sigma$, if the initial data (\ref{E1-2R1}) satisfies
$$
\|u_0(r)-u_{T^*}^{\pm}(0,r)\|_{\HH^1(\Omega_T)}+\|u_1(r)-\del_tu_{T^*}^{\pm}(0,r)\|_{\LL^2(\Omega_T)}\lesssim\eps,
$$
then there exists a positive constant $T$ depending on the initial data such that equation (\ref{E1-2}) admits a radial symmetric solution $u(t,r)$ of the form
$$
u(t,r)=u_T^{\pm}(t,r)+w(t,r),
$$
where the set $\Omega_{T-t}:=\{r: r\in(0,\sigma(T-t)]\}$, and
$$
\|w(t,r)\|_{\HH^1(\Omega_{T-t})}\lesssim \eps (T-t).
$$

Moreover, the blowup time $T$ belongs to $[T^*-\delta,T^*+\delta]$ for the positive constant $\delta\ll1$.

\end{itemize}
\end{theorem}

\paragraph{Sketch the proof of Theorem 1.1.}

Since the radially symmetric membranes equation (\ref{E1-2}) is a quasilinear wave equation with singular coefficients, it is hard to find explicit singular solutions \cite{Hop1,Hop2}. Thanks to the structure of nonlinear terms in (\ref{E1-2}), we can set 
$$u(t,r)=(T-t)\phi({r\over T-t}),\quad for\quad t<T,$$
then the radially symmetric membranes equation (\ref{E1-2}) is reduced into an ODE as follows
$$
\rho(1-\rho^2-\phi^2)\phi''+\phi'-\phi'\phi^2+2\rho\phi(\phi')^2+(1-\rho^2)(\phi')^3=0.
$$

Furthermore, let
$$
1-\rho^2-\phi^2=0,
$$
then it holds
$$
\phi'-\phi'\phi^2+2\rho\phi(\phi')^2+(1-\rho^2)(\phi')^3=0.
$$
Thus the radially symmetric membranes equation (\ref{E1-2}) admits two explicit solutions
\bel{A1-1}
u_T^{\pm}(t,r)=\pm(T-t)\sqrt{1-(\frac{r}{T-t})^2},\quad\quad (t,r)\in(0,T)\times[0,T-t],
\ee
which are lightlike solutions and break down at $t=T$ in the sense that
\bel{A1-2}
\partial_{rr}u_T^{\pm}(t,r)|_{r=0}\rightarrow+\infty,~~as~~t\rightarrow T^{-}.
\ee

In what follows, we consider the dynamical behavior near two explicit self-similar solutions. If two explicit self-similar solutions are nonlinearly stable, then the dynamical behavior of them are as attractors. Otherwise, there may exist the bifurcation phenomenon. 
We linearize the radially symmetric membranes equation (\ref{E1-2}) around two explicit self-similar solutions $u_T^{\pm}(t,r)$ in the similarity coordinates, then we get the linear equation 
$$
v_{\tau\tau}+3v_{\tau}-4v=0,
$$
which admits two eigenvalues $4$ and $-1$. Obviously, the positive eigenvalue $4$ is an unstable eigenvalue, which may cause the unstable phenomenon (see Definition 2.1 for mode unstable or stable).  Luckily, we find if we give a small perturbation to solutions (\ref{A1-1}), then the new linear equation only admits eigenvalues $\nu$ satisfying $Re~\nu<0$.
More precisely, let the small perturbation be of the form
$$
v^{(0)}(t,r)=(1-\kappa)u_T^{\pm}(t,r),
$$
where positive parameters $\kappa\in(\kappa_{\eps,\sigma},1)$ and $T\in[T^*-\delta,T^*+\delta]$ with $0<\delta\ll1$ and
$$
\kappa_{\eps,\sigma}:=1-(T\sigma)^{-{1\over2}}\Big(1-({\sigma T\over T^*})^2\Big)^{{1\over2}}\Big[1+\sigma+T\Big(1-({\sigma T\over T^*})^2\Big)\Big]^{-1}\eps,
$$
then the linearized equation (around $u_T^{\pm}(t,r)+v^{(0)}(t,r)$) in the similarity coordinates is 
$$
\aligned
\Big(1+(\kappa^2-1)\rho^2\Big)v_{\tau\tau}&+\Big(4\kappa^2-1+(\kappa-1)^2\rho^2\Big)v_{\tau}-(1-\kappa^2)(1-\rho^2)^2v_{\rho\rho}\\
&+2\rho(1-\kappa^2)(1-\rho^2)v_{\tau\rho}-\rho^{-1}(1-\kappa^2)(1-\rho^2)v_{\rho}-4\kappa^2v=0,
\endaligned
$$
which is well-posedness, and the corresponding eigenvalues of it satisfies $Re~\nu<0$. Here we can not follow the quasi-solution method given in \cite{Cos2,Cos3} to deal with our case. This is because there are double roots in solving the difference equation, and the theorem of Poincar\'{e} can't be used.
Thus, we have to carry out the analysis of this case by construction of a Newton's polygon \cite{Brie}.
Meanwhile, we notice that there is a non-autonomous term $-2\kappa(1-\kappa^2)(1-\rho^2)^{1\over2}$ in the nonlinear equation.
Since $\kappa\sim1$ and $\rho\in(0,\sigma]$ with $0<\sigma<1$, the non-autonomous term admits the property
$$
2\kappa(1-\kappa^2)(1-\rho^2)^{-\frac{1}{2}}\sim\eps_0\ll1.
$$
Hence, we should solve a non-autonomous quasilinear damped wave equation inside a strictly proper subset of the backward lightcone.
By construction of a suitable Nash-Moser iteration scheme \cite{Yan,Yan1,YZ,ZY}, we construct a solution $u(t,r)\in\HH^1(\Omega_{T-t})$ of the radially symmetric membranes equation (\ref{E1-2}) satisfies
$$
u(t,r)-u_T^{\pm}(t,r)=(1-\kappa)u_T^{\pm}(t,r)+w^{\infty}(t,r),\quad (t,r)\in(0,T)\times(0,\sigma(T-t)],
$$
where $w^{\infty}(t,r)\in\HH^2(\Omega_{T-t})$ is a small solution of a non-autonomous quasilinear damped wave equation, and the explicit form of it is given in (\ref{YAE4-19}). Furthermore, by noticing (\ref{A1-2}), we obtain nonlinearly stable of them in $\HH^1(\Omega_{T-t})$ with the set $$\Omega_{T-t}:=\{r: r\in(0,\sigma(T-t)]\}.$$ Here we impose the boundary condition $w^{\infty}(t,r)|_{r\in\del\Omega_{T-t}}=w_r^{\infty}(t,r)|_{r\in\del\Omega_{T-t}}=0$.

\paragraph{Notations.}
Thoughout this paper,
we denote the open ball in $\mathbb{R}^3$ by $\mathbb{B}_R^3$ centered at zero with radius $R>0$. When $R=1$, we write $\mathbb{B}^3$.
$\mathbb{N}$ is the natural numbers $\{1,2,3,\ldots\}$. $\mathbb{Z}$ is the integer number.
The symbol $a\lesssim b$ means that there exists a positive constant $C$ such that $a\leq Cb$. $(a,b)^T$ denotes the column vector in $\mathbb{R}^2$.
$a\simeq b$ implies that $a\lesssim b$ and $b\lesssim a$. $\sigma(\mathcal{A})$ and $\sigma_p(\mathcal{A})$ are the spectrum and point spectrum of the closed linear operator $\mathcal{A}$, respectively. $R_{\mathcal{A}}(\nu):=(\nu-\mathcal{A})^{-1}$ for $\nu\notin\sigma(\mathcal{A})$. $Re$ denotes the real part of the complex number.

Furthermore, let $\sigma<1$, $l\geq1$ and $\Omega:=(0,\sigma]$, we denote the usual norm of Sobolev space $\HH^l(\Omega)$ by $\|\cdot\|_{\HH^l}$ for convenience.
The space $\LL^2((0,\infty);\HH^l(\Omega)$ is equipped with the norm 
$$
\|v\|^2_{\LL^2((0,\infty);\HH^l)}:=\int_0^{\infty}\|v(t,\cdot)\|^2_{\HH^l}dt,
$$
and the function space $\Ccal^l_{1}:=\bigcap_{i= 0}^1\CC^i((0,\infty);\HH^{l-i})$ with the norm
$$
\|v\|^2_{\Ccal^l_{1}}:=\sup_{t\in(0,\infty)}\sum_{i= 0}^1\|\partial^{i}_{t}v\|^2_{\HH^{l-i}}.
$$

The organization of this paper is as follows. In Section 2, we give the existence of explicit self-similar solutions $u_T^{\pm}(t,r)$ and the linear mode unstable of them (see Definition 2.1 for the mode stable and unstable of solutions). In section 3, firstly, the well-posedness result for the linearized radially symmetric membranes equation around the initial approximation function is shown by the semigroup theory, then the analysis of spectrum of linearized operator is given by construction of a Newton's polygon. After that, we show the well-posedness result for the linearized radially symmetric membranes equation at the general approximation step.
In section 4, the nonlinearly stable of explicit self-similar solutions $u_T^{\pm}(t,r)$ is proven by contruction of the Nash-Moser iteration scheme.
One can see \cite{H,Moser,Nash} more details on this method.


\section{Linear mode unstable of lightlike self-similar solutions}\setcounter{equation}{0}
This section gives the proof of mode unstable for two explicit self-similar solutions of the radially symmetric membranes equation (\ref{E1-2}). Firstly, we show how to find two explicit self-similar solutions of equation (\ref{E1-2}) by the scaling invariant of (\ref{E1-3}) and the structure of nonlinear terms. Then the mode unstable of them are proved. 

\subsection{Two explicit lightlike self-similar solutions}
Self-similar solutions are invariant under the scaling (\ref{E1-3}), so we can set
$$
\aligned
&\rho=\frac{r}{T-t},\\
&u(t,r)=(T-t)\phi(\rho),
\endaligned
$$
where $T$ is a positive constant.

Inserting this ansatz into equation (\ref{E1-2}) by noticing
$$
\aligned
&\partial_tu(t,r)=-\phi(\rho)+\rho\phi'(\rho),\\
&\partial_{tt}u(t,r)=(T-t)^{-1}\rho^2\phi''(\rho),\\
&\partial_ru(t,r)=\phi'(\rho),\\
&\partial_{rr}u(t,r)=(T-t)^{-1}\phi''(\rho),\\
&\partial_{tr}u(t,r)=(T-t)^{-1}\rho\phi''(\rho),
\endaligned
$$
we obtain a quasilinear ordinary differential equation
\bel{E2-1}
\rho(1-\rho^2)\phi''+\phi'-\phi'\phi^2+2\rho\phi(\phi')^2-\rho\phi''\phi^2+(1-\rho^2)(\phi')^3=0.
\ee

Here we are only interested in smooth solutions in the backward lightcone of blowup point $(t,r)=(T,0)$, i.e. $\rho$ in the closed interval $[0,1]$. It satisfies
\bel{E2-1R1}
\frac{\partial^n}{\partial r^n}\phi(\frac{r}{T-t})|_{r=0}=(T-t)^{-n}\frac{d^n\phi}{dy^n}(0),
\ee
where if the $n$-th derivative is even, there is $\frac{d^n\phi}{dy^n}(0)\neq0$. If $n$ is odd, there is $\frac{d^n\phi}{dy^n}(0)=0$. This condition is different with wave map (e.g. see \cite{BB}).
One can see that every self-similar solution $\phi(\rho)\in\mathbb{C}^{\infty}[0,1]$ describes a singularity developing in finite time from smooth initial data. 
By the structure of nonlinear terms in equation (\ref{E2-1}), it holds
\begin{equation}\label{E2-1R}
\rho(1-\rho^2-\phi^2)\phi''+\phi'-\phi'\phi^2+2\rho\phi(\phi')^2+(1-\rho^2)(\phi')^3=0.
\end{equation}

Let
\begin{equation*}
1-\rho^2-\phi^2=0,
\end{equation*}
then there is
\begin{equation*}
\phi'-\phi'\phi^2+2\rho\phi(\phi')^2+(1-\rho^2)(\phi')^3=0.
\end{equation*}
Thus we obtain two explicit solutions
\begin{equation}\label{E2-2}
\phi(\rho)=\pm\sqrt{1-\rho^2},
\end{equation}
which satisfies (\ref{E2-1R1}).

Consequently, two explicit self-similar solutions of (\ref{E1-2}) are
\bel{E2-3}
u_T^{\pm}(t,r)=\pm(T-t)\sqrt{1-(\frac{r}{T-t})^2},
\ee
which exhibit the smooth for all $0<t<T$, but which break down at $t=T$ in the sense that
$$
\aligned
\partial_{rr}u_T^{\pm}(t,r)|_{r=0}&=\pm((T-t)^2-r^2)^{-\frac{1}{2}}|_{r=0}\pm r^2((T-t)^2-r^2)^{-\frac{3}{2}}|_{r=0}\nonumber\\
&=\pm\frac{1}{T-t}\rightarrow+\infty,~~as~~t\rightarrow T^{-},
\endaligned
$$
and the dynamical behavior of them are as attractors.

On the other hand, from the form of $u_T^{\pm}(t,r)$ in (\ref{E2-3}), it requires that
\begin{equation*}
1-(\frac{r}{T-t})^2\geq0.
\end{equation*}
So we consider the dynamical behavior of self-similar solutions in
of the backward lightcone
$$
\mathcal{B}_T:=\{(t,r):t\in(0,T),~~r\in[0,T-t]\}.
$$

\begin{remark}
Obviously, self-similar solutions (\ref{E2-3}) are cycloids. The sphere begins to expand until it starts to shrink and eventually collapses to a point in a finite time $\tilde{T}=T-r_0$, i.e. $u^{\pm}(\tilde{T},r_0)=0$.
Here $r_0$ is a fixed positive constant in the backward lightcone.
\end{remark}

\subsection{Mode unstable of self-similar solutions $u_T^{\pm}(t,r)$}

We introduce the similarity coordinates
\begin{equation}\label{E2-5}
\tau=-\log(T-t)+\log T,~~~~\rho=\frac{r}{T-t},
\end{equation}
and we set
\begin{equation*}
\tilde{v}(\tau,\rho)=T^{-1}e^{\tau}u(T(1-e^{-\tau}),T\rho e^{-\tau}),
\end{equation*}
then equation (\ref{E1-2}) is transformed into
\bel{E2-4}
\aligned
\tilde{v}_{\tau\tau}-\tilde{v}_{\tau}-(1-\rho^2)\tilde{v}_{\rho\rho}&-\frac{1}{\rho}\tilde{v}_{\rho}+2\rho \tilde{v}_{\tau\rho}+\tilde{v}_{\rho}^2(\tilde{v}_{\tau\tau}+\tilde{v}_{\tau}-2\tilde{v})+\tilde{v}_{\rho\rho}
(\tilde{v}-\tilde{v}_{\tau})^2\\
&-2\tilde{v}_{\rho}\tilde{v}_{\tau\rho}(\tilde{v}_{\tau}-\tilde{v})+\frac{1}{\rho}\tilde{v}_{\rho}(\tilde{v}_{\tau}-\tilde{v})^2+\frac{1}{\rho}(\rho^2-1)\tilde{v}_{\rho}^3=0.
\endaligned
\ee

In the similarity coordinates (\ref{E2-5}), the blowup time $T$ is changed to $\infty$. So the stability of blowup solutions for the radially symmetric membranes equation  (\ref{E1-2}) as $t\rightarrow T^{-}$ is transformed into the asymptotic stability for quasilinear wave equation (\ref{E2-4}) as $\tau\rightarrow\infty$. We now consider the problem of mode unstable for linear equation of (\ref{E2-4}).

Let the solution of (\ref{E2-4}) takes the form
\begin{eqnarray}\label{E2-4R}
\tilde{v}(\tau,\rho)&=&\phi(\rho)+v(\tau,\rho),
\end{eqnarray}
where $\phi(\rho)$ is defined in (\ref{E2-2}).

Inserting (\ref{E2-4R}) into equation (\ref{E2-4}), it holds
\bel{E2-6r}
\aligned
(1+\phi'^2)v_{\tau\tau}&-(1-\rho^2-\phi^2)v_{\rho\rho}-(1-\phi'^2+2\phi''\phi+\frac{2}{\rho}\phi'\phi)v_{\tau}+2(\phi\phi'+\rho)v_{\tau\rho}\\
&-\frac{1}{\rho}\Big(1+4\rho\phi'\phi-3(\rho^2-1)\phi'^2-\phi^2\Big)v_{\rho}+\frac{1}{\rho}(-2\rho\phi'^2+2\rho\phi''\phi+2\phi'\phi)v=f(\phi,v),
\endaligned
\ee
where 
$$
\aligned
f(\phi,v):=&2\phi v_{\rho}^2-v_{\rho}(v_{\tau\tau}+v_{\tau}-2v)(v_{\rho}+2\phi')-\phi''(v-v_{\tau})^2\\
&-v_{\rho\rho}\Big[(v-v_{\tau})^2+2\phi(v-v_{\tau})\Big]+2\phi'v_{\tau\rho}(-v+v_{\tau})-2\phi v_{\rho}v_{\rho\tau}\\
&+2v_{\rho}v_{\rho\tau}(-v+v_{\tau})-\frac{1}{\rho}v_{\rho}\Big[(v-v_{\tau})^2+2\phi(v-v_{\tau})\Big]-\frac{1}{\rho}\phi'(v-v_{\tau})^2\\
&-(\rho-\frac{1}{\rho})(v_{\rho}^3+3\phi'v_{\rho}^2).
\endaligned
$$

It follows from the exact form of $\phi(\rho)$ in (\ref{E2-2}) that
$$
1-\rho^2-\phi^2=0,
$$
$$
\phi\phi'+\rho=0,
$$
thus, equation  (\ref{E2-6r}) is reduced into 
\bel{E2-7RR}
\aligned
v_{\tau\tau}+3v_{\tau}-4v=f(\rho,v),
\endaligned
\ee
where
\bel{E2-7RR1}
\aligned
f(\rho,v)=&2(1-\rho^2)^{\frac{1}{2}}v_{\rho}^2-v_{\rho}(v_{\tau\tau}+v_{\tau}-2v)(v_{\rho}-2\rho(1-\rho^2)^{-\frac{1}{2}})\\
&+(1-\rho^2)^{-\frac{3}{2}}(v-v_{\tau})^2-v_{\rho\rho}\Big[(v-v_{\tau})^2+2(1-\rho^2)^{\frac{1}{2}}(v-v_{\tau})\Big]\\
&-2\rho(1-\rho^2)^{-\frac{1}{2}}v_{\tau\rho}(-v+v_{\tau})-2(1-\rho^2)^{\frac{1}{2}} v_{\rho}v_{\rho\tau}+2v_{\rho}v_{\rho\tau}(-v+v_{\tau})\\
&-\frac{1}{\rho}v_{\rho}\Big[(v-v_{\tau})^2+2(1-\rho^2)^{\frac{1}{2}}(v-v_{\tau})\Big]+(1-\rho^2)^{-\frac{1}{2}}(v-v_{\tau})^2\\
&-(\rho-\rho^{-1})(v_{\rho}^3-3\rho(1-\rho^2)^{-\frac{1}{2}}v_{\rho}^2).
\endaligned
\ee

It is easy to see equation (\ref{E2-7RR}) is loss of hyperbolicity. 
The linear equation of it is 
\bel{E2-7r}
v_{\tau\tau}+3v_{\tau}-4v=0.
\ee

Set
$$
v(\tau,\rho)=e^{\nu\tau}u_{\nu},
$$
then it leads to an eigenvalue problem
\begin{equation}\label{E2-4R1}
(\nu^2+3\nu-4)u_{\nu}=0.
\end{equation}

As in \cite{Cos0,Cos1,Cos2, Cos3}, we introduce  the definition of mode stable or unstable for the solution $u_{\nu}$ of (\ref{E2-4R1}).

\begin{definition}
A non-zero smooth solution $u_{\nu}$ of (\ref{E2-4R1}) is called mode stable if $Re~\nu<0$ holds. The eigenvalue $\nu$ is called a stable eigenvalue. Otherwise, if $Re~\nu\geq0$, the non-zero smooth solution $u_{\nu}$ of (\ref{E2-4R1}) is called mode unstable. Then $\nu$ is called an unstable eigenvalue.
\end{definition}

From (\ref{E2-4R1}), the linear equation (\ref{E2-7r}) admits two eigenvalues $\nu=4$ and $\nu=-1$ .
By Definition 2.1, we know that two explicit self-similar solutions $u_T^{\pm}(t,r)$ given in (\ref{E2-3}) of timelike extremal hypersurfaces equation (\ref{E1-2}) are mode unstable inside the backward lightcone $\mathcal{B}_T$.

\section{Well-posedness of linearized time evolution}\setcounter{equation}{0}
We have shown that two explicit lightlike self-similar solutions $u_T^{\pm}(t,r)$ in (\ref{E2-3}) are mode unstable. There is an unstable eigenvalue in linear equation. 
It is natural to investigate whether explicit singular solutions given in (\ref{E2-3}) are nonlinearly stable or unstable.
In order to get a positive answer, we should overcome two 
difficulties. One is to deal with an unstable eigenvalue in linear equation. Another difficulty is to solve the ``loss of derivatives'' in nonlinear equation. Luckily, we find that if we choose a suitable initial approximation function $v^{(0)}(\tau,\rho)$, then linearizing quasilinear wave equation (\ref{E2-7RR}) around $v^{(0)}(\tau,\rho)$, we get a linear equation only admitted stable eigenvalues. 
Meanwhile, it causes a small error term. Following the Nash-Moser iteration process \cite{Yan}, we can obtain a desired solution of equation (\ref{E2-7RR}) as follows
$$
v(\tau,\rho)=v^{(0)}(\tau,\rho)+w^{(0)}(\tau,\rho)+\sum_{k=1}^{\infty}h^{(k)}(\tau,\rho)+\eps w_0(\rho)-\eps (e^{-\tau}-1)w_1(\rho),
$$
where $w^{(0)}(\tau,\rho)\sim\eps$ and $\sum_{k=1}^{\infty}h^{(k)}\sim\Ocal(\eps^p)$ with a fixed constant $p\geq2$ in some Sobolev space. Meanwhile, we can overcome the ``loss of derivatives'' of nonlinear equation by means of Nash-Moser iteration.

\subsection{The $\CC^0$-semigroup of linearized operator at the initial aproximation step}
Let positive constants $\kappa\sim1$ and $\kappa<1$.
We choose the initial approximation function
$$
v^{(0)}(\tau,\rho)=(\kappa-1)\phi(\rho),
$$
where $\phi(\rho)$ is defined in (\ref{E2-2}).

Linearizing equation (\ref{E2-7RR}) around $v^{(0)}$, then we get the linearized operator as follows
\bel{YYE3-1}
\aligned
\Lcal^{(0)}(v):=\Big(1+(\kappa^2-1)\rho^2\Big)v_{\tau\tau}&-(1-\kappa^2)(1-\rho^2)^2v_{\rho\rho}+\Big(4\kappa^2-1+(\kappa-1)^2\rho^2\Big)v_{\tau}\\
&\quad +2\rho(1-\kappa^2)(1-\rho^2)v_{\tau\rho}-(1-\kappa^2)(1-\rho^2)\rho^{-1}v_{\rho}-4\kappa^2v.
\endaligned
\ee

In order to process well-posedness result, we introduce the radial Sobolev functions $\hat{v}:\mathbb{B}_R^3\rightarrow\mathbb{C}$, i.e. $\hat{v}(x)=v(|x|)$ for all $x\in\mathbb{B}_R^3$. Here $v:(0,R)\rightarrow\mathbb{C}$.
Following \cite{D2}, for any $n>1$, we define
\begin{equation*}
v\in\mathbb{H}_{\rho}^n(\mathbb{B}_R^3),
\end{equation*}
if and only if
\begin{equation*}
\hat{v}\in\mathbb{H}^n(\mathbb{B}_R^3)=\mathbb{W}^{n,2}(\mathbb{B}_R^3),
\end{equation*}
with the norm
\begin{equation*}
\|\hat{v}\|_{\mathbb{H}^n(\mathbb{B}_R^3)}=\|\hat{v}\|_{\mathbb{W}^{n,2}(\mathbb{B}_R^3)}.
\end{equation*}

Obviously, $\mathbb{H}_{\rho}^n(\mathbb{B}_R^3)$ is a Banach space. For convenience, throughout this paper we do not distinguish between
$v$ and $\hat{v}$.

Define the Hilbert spaces
$$
\mathcal{H}=\mathbb{H}_{\rho}^2(\mathbb{B}_R^3)\times\mathbb{H}_{\rho}^1(\mathbb{B}_R^3),
$$
and
$$
\mathcal{H}_0=\mathbb{H}_{\rho}^1(\mathbb{B}_R^3)\times\mathbb{L}_{\rho}^2(\mathbb{B}_R^3),
$$
with the induced norms
$$
\|v\|^2_{\mathcal{H}}:=\|v_1\|^2_{\mathbb{H}_{\rho}^2(\mathbb{B}_R^3)}+\|v_2\|^2_{\mathbb{H}_{\rho}^1(\mathbb{B}_R^3)},
$$
and
$$
\|v\|^2_{\mathcal{H}_0}:=\|v_1\|^2_{\mathbb{H}_{\rho}^1(\mathbb{B}_R^3)}+\|v_2\|^2_{\mathbb{L}_{\rho}^2(\mathbb{B}_R^3)},
$$
respectively.

Moreover, it holds
$$
\|v(|\cdot|)\|_{\mathbb{L}_{\rho}^2(\mathbb{B}^3)}^2:=\|\rho v\|^2_{\mathbb{L}^2(0,1)},
$$

$$
\|v(|\cdot|)\|_{\mathbb{H}_{\rho}^1(\mathbb{B}^3)}^2\simeq\|v\|_{\mathbb{L}^2(0,1)}^2+\|\rho\partial_{\rho}v\|_{\mathbb{L}^2(0,1)}^2,~~\forall v\in\mathbb{H}_{\rho}^1(\mathbb{B}^3),
$$
and
$$
\|v(|\cdot|)\|_{\mathbb{H}_{\rho}^2(\mathbb{B}^3)}^2\simeq\|v\|_{\mathbb{L}^2(0,1)}^2+\sum_{i=1}^2\|(\cdot)^{i}v^{(i)}\|_{\mathbb{L}^2(0,1)}^2,~~\forall v\in\mathbb{H}_{\rho}^2(\mathbb{B}^3),
$$
where $\mathbb{B}^3$ denotes a ball with radius $1$ in $\mathbb{R}^3$. 

We now consider the well-posedness of linear problem
\bel{Y1-1}
\aligned
&\Lcal^{(0)}v=0,\quad \tau>0,\\
& v(0,\rho)=v_0,\quad v_{\tau}(0,\rho)=v_1
\endaligned
\ee
inside the backward lightcone set
$$
\overline{\mathcal{B}}_{\sigma}:=\{(\tau,\rho):\tau\in(0,+\infty),\quad \rho\in(0,\sigma]\}.
$$

More precisely, by (\ref{YYE3-1}), the linear problem (\ref{Y1-1}) has the form
\bel{E2-7}
\aligned
\Big(1+(\kappa^2-1)\rho^2\Big)v_{\tau\tau}&+\Big(4\kappa^2-1+(\kappa-1)^2\rho^2\Big)v_{\tau}-(1-\kappa^2)(1-\rho^2)^2v_{\rho\rho}\\
&+2\rho(1-\kappa^2)(1-\rho^2)v_{\tau\rho}-\rho^{-1}(1-\kappa^2)(1-\rho^2)v_{\rho}-4\kappa^2v=0,
\endaligned
\ee
with the initial data
\bel{E2-8R1}
v(0,\rho):=v_0^T(\rho)=T^{-1}u_0(T\rho)-(\kappa-1)\phi(\rho),
\ee
\bel{E2-8R2}
v_{\tau}(0,\rho):=v_1^T(\rho)=T^{-1}u_0(T\rho)+(1-\rho)u_1(T\rho)-(\kappa-1)\phi(\rho).
\ee

Since $\kappa\sim1$ and $\kappa<1$, the coefficient of $v_{\tau}$ is positive, i.e. 
$$
4\kappa^2-1+(\kappa-1)^2\rho^2>0.
$$
Thus  (\ref{E2-7}) is a linear  damped wave equation with variable coefficient $\rho$, which corresponds to a non-selfadjoint linear wave operator. Since there is a term $\rho^{-1}$ in (\ref{E2-7}), it has a singular point at $\rho=0$, and it is not easy to study the specturm problem of (\ref{E2-7}).

We now rewrite the linear equation (\ref{E2-7}) as an evolution equation. Let $\psi=(v,w)$ and $w=v_{\tau}$. Then the linear equation (\ref{E2-7}) is equivalent to
\begin{equation}\label{E3-32}
\frac{d}{d\tau}\psi(\tau)=\mathcal{A}\psi(\tau),~~\tau>0,
\end{equation}
with the initial data
$$
\psi_0:=\psi(0,\rho)=(v_0^T(\rho),v_1^T(\rho)),
$$
where $v_0^T(\rho),v_1^T(\rho)$ are given by (\ref{E2-8R1})-(\ref{E2-8R2}), the operator $\mathcal{A}$ is independent of $\tau$, it has the form
\bel{E3-32R1}
\aligned
\mathcal{A}=\left(
\begin{array}{cccc}
0&1\\
\frac{(1-\kappa^2)(1-\rho^2)^2\partial_{\rho\rho}+(1-\kappa^2)(1-\rho^2)\rho^{-1}\partial_{\rho}+4\kappa^2}{1+(\kappa^2-1)\rho^2}&\frac{-[4\kappa^2-1+(\kappa^2-1)^2\rho^2]-2(1-\kappa^2)\rho(1-\rho^2)\partial_{\rho}}{1+(\kappa^2-1)\rho^2}
\end{array}
\right).
\endaligned
\ee

We define
\bel{E3-32R2}
\aligned
\mathcal{A}_0:=\left(
\begin{array}{cccc}
0&1\\
\varpi_1(\rho)&\varpi_2(\rho)
\end{array}
\right),
\endaligned
\ee
and
\bel{E3-34}
\mathcal{A}_1=\left(
\begin{array}{cccc}
0&0\\
\frac{4\kappa^2}{1+(\kappa^2-1)\rho}+1-(\rho^{-1}a(\rho)+b(\rho))\partial_{\rho}&0
\end{array}
\right),
\ee
where
\bel{E3-34R1}
\aligned
\varpi_1(\rho)&=\frac{(1-\kappa^2)(1-\rho^2)^2\partial_{\rho\rho}+(1-\kappa^2)(1-\rho^2)\rho^{-1}\partial_{\rho}+(\rho^{-1}a(\rho)+b(\rho))\partial_{\rho}}{1+(\kappa^2-1)\rho^2}-1,\\
\varpi_2(\rho)&=\frac{-[4\kappa^2-1+(\kappa^2-1)^2\rho^2]-2(1-\kappa^2)\rho(1-\rho^2)\partial_{\rho}}{1+(\kappa^2-1)\rho^2},\\
a(\rho)&=-\frac{\rho^2(1-\kappa^2)(1-\rho^2)}{1+(\kappa^2-1)\rho^2}-\rho^2,\\
b(\rho)&=-\frac{4\rho(1-\kappa^2)(1-\rho^2)}{1+(\kappa^2-1)\rho^2}+\frac{2\rho(1-\kappa^2)^2(1-\rho^2)^2}{(1+(\kappa^2-1)\rho^2)^2}.
\endaligned
\ee

Thus it holds
\bel{E3-35}
\mathcal{A}=\mathcal{A}_0+\mathcal{A}_1.
\ee

Following \cite{D2}, we set
$$
\mathbb{D}(\mathcal{A}_0):=\{u=(u_1,u_2)\in\mathbb{C}^{\infty}(0,1)^2\cap\mathcal{H}: u_1\in\mathbb{C}^{2}(0,1),~~u_1''(0)=0,~~u_2\in\mathbb{C}^{
1}(0,1)\},
$$
where for some positive constants $c_0$ and $c_1$, and it holds
$$
u_i=c_0\rho v_i(\rho)+c_1\rho^2v_i'(\rho),~~i=1,2.
$$
Note that $\mathbb{C}^{\infty}(\overline{\mathbb{B}^3})$ is dense in $\mathbb{H}^n(\mathbb{B}^3)$. So $\mathbb{C}_e^{\infty}[0,1]^2$ is dense in $\mathcal{H}$, where
$$
\mathbb{C}_e^{\infty}[0,1]^2:=\{v\in\mathbb{C}^{\infty}[0,1]^2: v^{(2j+1)}(0)=0,~~j=0,1,2,3,\ldots\}\subset\mathbb{D}(\mathcal{A}_0).
$$

Furthermore, let
$$
\mathcal{A}_0u=([\mathcal{A}_0u]_1,[\mathcal{A}_0u]_2)=(u_2,\varpi_1(\rho)u_1+\varpi_2(\rho)u_2).
$$

\begin{lemma}
There is $\mathcal{A}_0u\in\mathcal{H}$ for any $u\in\mathbb{D}(\mathcal{A}_0)$.
\end{lemma}
\begin{proof}
The first component of $\mathcal{A}_0u$ is $[\mathcal{A}_0u]_1=u_2$. Since $u\in\mathbb{D}(\mathcal{A}_0)$, $u_2\in\mathbb{H}_{\rho}^2(\mathbb{B}_R^3)$. By (\ref{E3-32R1}) and (\ref{E3-34R1}),
the second component of $\mathcal{A}_0u$ is
\begin{eqnarray*}
[\mathcal{A}_0u]_2 &=&\varpi_1(\rho)u_1+\varpi_2(\rho)u_2\\
&=&\frac{(1-\kappa^2)(1-\rho^2)^2(u_1)_{\rho\rho}+(1-\kappa^2)(1-\rho^2)\rho^{-1}(u_1)_{\rho}+(\rho^{-1}a(\rho)+b(\rho))(u_1)_{\rho}}{1+(\kappa^2-1)\rho^2}-u_1\\
&&\quad+\frac{-[4\kappa^2-1+(\kappa^2-1)^2\rho^2]u_2-2(1-\kappa^2)\rho(1-\rho^2)(u_2)_{\rho}}{1+(\kappa^2-1)\rho^2}.
\end{eqnarray*}

Using de l'H\^{o}pital's rule and noticing $v\in\mathbb{D}(\mathcal{A}_0)$, it holds
\bel{E3-33}
\lim_{\rho\rightarrow0^+}\frac{(u_1)_{\rho}(\rho)}{\rho}=\lim_{\rho\rightarrow0^+}\frac{(u_1)_{\rho}(\rho)-(u_1)_{\rho}(0)}{\rho-0}=(u_1)_{\rho\rho},
\ee
which combing with $1+(\kappa^2-1)\rho^2\in[\kappa^2,1]$ gives that $[\mathcal{A}_0u]_2\in\mathbb{H}_{\rho}^1(\mathbb{B}_R^3)$. Here $\kappa\sim1$ and $\kappa<1$.
\end{proof}

We introduce the sesquilinear form
$$
(a|b)_{\mathbb{L}^2(0,1)}=\int_0^1a_1\bar{b}_1d\rho+\int_0^1a_2\bar{b}_2d\rho,\quad \forall a(\rho)=(a_1,a_2),\quad b(\rho)=(b_1,b_2)\in\mathbb{C}^2[0,1],
$$
which means that the norm of $\mathbb{L}^2(0,1)$ is equivalent to $\sqrt{(a|a)_{\mathbb{L}^2(0,1)}}$.

\begin{lemma}
The operator $\mathcal{A}_0$ defined in (\ref{E3-32R2}) is a closed and densely defined linear dissipative operator in $\mathcal{H}$.
\end{lemma}
\begin{proof}
It is easy to check that $\mathcal{A}_0$ is a densely defined and closed linear operator in $\mathcal{H}$. Here we only need to prove that $\mathcal{A}_0$ is dissipative, i.e. $$(\mathcal{A}_0u|u)\leq0.$$

For any $u=(u_1,u_2)\in\mathbb{D}(\mathcal{A}_0)$, by (\ref{E3-32R2}) and (\ref{E3-34R1}), it holds
\bel{E3-36}
\aligned
Re(\mathcal{A}_0u|u)_{\mathcal{H}}&=Re(u_2,u_1)_{\mathbb{H}_{\rho}^2}+Re([\mathcal{A}_0u]_2,u_2)_{\mathbb{H}_{\rho}^1},\\
&=Re(u_2,u_1)_{\mathbb{L}^2(0,1)}+Re(\rho(u_1)_{\rho},\rho(u_2)_{\rho})_{\mathbb{L}^2(0,1)}+Re(\rho^2(u_1)_{\rho\rho},\rho^2(u_2)_{\rho\rho})_{\mathbb{L}^2(0,1)}\\
&\quad+Re([\mathcal{A}_0u]_2,u_2)_{\mathbb{L}^2(0,1)}+Re(\rho([\mathcal{A}_0u]_2)_{\rho},\rho(u_2)_{\rho})_{\mathbb{L}^2(0,1)}\\
&=Re\int_0^1u_2\bar{u}_1d\rho+Re\int_0^1\rho^2(u_2)_{\rho}(\bar{u}_1)_{\rho}d\rho+Re\int_0^1\rho^4(u_2)_{\rho\rho}(\bar{u}_1)_{\rho\rho}d\rho\\
&\quad+Re\int_0^1(\varpi_1(\rho)u_1)\bar{u}_2d\rho+Re\int_0^1(\varpi_2(\rho)u_2)\bar{u}_2d\rho
+Re\int_0^1\rho^2(\varpi_1(\rho)u_1)_{\rho}(\bar{u}_2)_{\rho}d\rho\\
&\quad+Re\int_0^1\rho^2(\varpi_2(\rho)u_2)_{\rho}(\bar{u}_2)_{\rho}d\rho,
\endaligned
\ee
where
\bel{E3-36Y1}
\aligned
&Re\int_0^1(\varpi_1(\rho)u_1)\bar{u}_2d\rho=Re\int_0^1\frac{(1-\kappa^2)(1-\rho^2)^2(u_1)_{\rho\rho}\bar{u}_2}{1+(\kappa^2-1)\rho^2}d\rho\\
&\quad\quad\quad\quad\quad\quad\quad\quad\quad\quad+Re\int_0^1\frac{(1-\kappa^2)(1-\rho^2)\rho^{-1}(u_1)_{\rho}\bar{u}_2+(\rho^{-1}a(\rho)+b(\rho))(u_1)_{\rho}\bar{u}_2}{1+(\kappa^2-1)\rho^2}-u_1\bar{u}_2d\rho,\\
&Re\int_0^1(\varpi_2(\rho)u_2)\bar{u}_2d\rho=Re\int_0^1\frac{-[4\kappa^2-1+(\kappa^2-1)^2\rho^2]u_2\bar{u}_2-2(1-\kappa^2)\rho(1-\rho^2)(u_2)_{\rho}\bar{u}_2}{1+(\kappa^2-1)\rho^2}d\rho.
\endaligned
\ee

Note that $u=(u_1,u_2)\in\mathbb{D}(\mathcal{A}_0)$. On one hand, direct computation shows that
\bel{E3-35R1}
\aligned
Re\int_0^1\frac{(1-\kappa^2)(1-\rho^2)^2(u_1)_{\rho\rho}\bar{u}_2}{1+(\kappa^2-1)\rho^2}d\rho&=-Re\int_0^1\frac{(1-\kappa^2)(1-\rho^2)^2(u_1)_{\rho}(\bar{u}_2)_{\rho}}{1+(\kappa^2-1)\rho^2}d\rho\\
&\quad-Re\int_0^1\frac{-4\rho(1-\kappa^2)(1-\rho^2)(u_1)_{\rho}\bar{u}_2}{1+(\kappa^2-1)\rho^2}d\rho\\
&\quad-Re\int_0^1\frac{2\rho(1-\kappa^2)^2(1-\rho^2)^2(u_1)_{\rho}\bar{u}_2}{(1+(\kappa^2-1)\rho^2)^2}d\rho.
\endaligned
\ee
By (\ref{E3-34R1}) and using the de l'H\^{o}pital's rule (\ref{E3-33}), it holds
\bel{E3-35R2}
\aligned
&Re\int_0^1\frac{(1-\kappa^2)(1-\rho^2)\rho^{-1}(u_1)_{\rho}\bar{u}_2+(\rho^{-1}a(\rho)+b(\rho))(u_1)_{\rho}\bar{u}_2}{1+(\kappa^2-1)\rho^2}d\rho\\
&\quad=Re\int_0^1\frac{(1-\kappa^2)(1-\rho^2)(u_1)_{\rho}(\bar{u}_2)_{\rho}+a(\rho)(u_1)_{\rho}(\bar{u}_2)_{\rho}}{1+(\kappa^2-1)\rho^2}d\rho+Re\int_0^1\frac{b(\rho)(u_1)_{\rho}\bar{u}_2}{1+(\kappa^2-1)\rho^2}d\rho\\
&\quad=Re\int_0^1\frac{(1-\kappa^2)(1-\rho^2)^2(u_1)_{\rho}(\bar{u}_2)_{\rho}}{1+(\kappa^2-1)\rho^2}d\rho-Re\int_0^1\rho^2(u_2)_{\rho}(u_1)_{\rho}d\rho\\
&\quad\quad+Re\int_0^1(-\frac{4\rho(1-\kappa^2)(1-\rho^2)}{1+(\kappa^2-1)\rho^2}+\frac{2\rho(1-\kappa^2)^2(1-\rho^2)^2}{(1+(\kappa^2-1)\rho^2)^2})\frac{(u_1)_{\rho}\bar{u}_2}{1+(\kappa^2-1)\rho^2}d\rho.
\endaligned
\ee
So by (\ref{E3-35R1})-(\ref{E3-35R2}), we have
\bel{E3-36Y2}
Re\int_0^1(\varpi_1(\rho)u_1)\bar{u}_2d\rho=-Re\int_0^1u_2\bar{u}_1d\rho-Re\int_0^1\rho^2(u_2)_{\rho}(u_1)_{\rho}d\rho.
\ee

On the other hand, note that $\kappa\sim1$ and $\kappa<1$, there exists a positive constant $\varsigma_0$ such that
\bel{E3-36Y3}
\aligned
&Re\int_0^1(\varpi_2(\rho)u_2)\bar{u}_2d\rho\\
&\quad=Re\int_0^1\frac{-[4\kappa^2-1+(\kappa^2-1)^2\rho^2]u_2\bar{u}_2-2(1-\kappa^2)\rho(1-\rho^2)(u_2)_{\rho}\bar{u}_2}{1+(\kappa^2-1)\rho^2}d\rho\\
&\quad=Re\int_0^1\left(-\kappa^{-2}(4\kappa^2-1+(\kappa^2-1)^2\rho^2)+\frac{(1-\kappa^2)(1-3\rho^2)}{1+(\kappa^2-1)\rho^2}+\frac{2\rho^2(1-\kappa^2)^2(1-\rho^2)}{(1+(\kappa^2-1)\rho^2)^2}\right)u_2\bar{u}_2d\rho\\
&\quad\leq-2Re\int_0^1\kappa^{-2}(2\kappa^2-1)u_2\bar{u}_2d\rho\\
&\quad\leq-\varsigma_0\|u_2\|^2_{\mathbb{L}^2([0,1])}<0.
\endaligned
\ee
Furthermore, by (\ref{E3-36Y2}) and (\ref{E3-36Y3}), there exists a positive constant $\varsigma_1$ such that
\bel{E3-36Y4}
\aligned
Re\int_0^1(\varpi_1(\rho)u_1)\bar{u}_2d\rho&+Re\int_0^1(\varpi_2(\rho)u_2)\bar{u}_2d\rho\\
&\leq-Re\int_0^1u_2\bar{u}_1d\rho-Re\int_0^1\rho^2(u_2)_{\rho}(u_1)_{\rho}d\rho-\varsigma_1\|u_2\|^2_{\mathbb{L}^2([0,1])}\\
&\leq-Re\int_0^1u_2\bar{u}_1d\rho-Re\int_0^1\rho^2(u_2)_{\rho}(u_1)_{\rho}d\rho.
\endaligned
\ee

Similarly, by direct computations, it holds
\bel{E3-36Y5}
Re\int_0^1\rho^2(\varpi_1(\rho)u_1)_{\rho}(\bar{u}_2)_{\rho}d\rho+Re\int_0^1\rho^2(\varpi_2(\rho)u_2)_{\rho}(\bar{u}_2)_{\rho}d\rho\leq-Re\int_0^1\rho^4(u_2)_{\rho\rho}(\bar{u}_1)_{\rho\rho}d\rho.
\ee

Hence, it follows from (\ref{E3-36}), (\ref{E3-36Y4}) and (\ref{E3-36Y5}) that
$$
Re(\mathcal{A}_0u|u)_{\mathcal{H}}\leq 0.
$$
This completes the proof.
\end{proof}

\begin{lemma}
The operator $\mathcal{A}_0$ defined in (\ref{E3-32R2}) is invertible in $\mathcal{H}$. Moreover, the operator $\mathcal{A}_0$ generates a $\mathbb{C}_0$-semigroup $(\textbf{S}_0(\tau))_{\tau\geq0}$ in $\mathcal{H}$.
\end{lemma}
\begin{proof}
In order to show the existence of $\mathcal{A}_0^{-1}$, we need to prove the operator $\mathcal{A}_0$ are injective and surjective.
We first show $\mathcal{A}_0$ is injective. Let $u=(u_1,u_2)\in\mathbb{D}(\mathcal{A}_0)$ such that $\mathcal{A}_0u=0$. Then $u_2=0$ and
$$
(1-\kappa^2)(1-\rho^2)^2(u_1)_{\rho\rho}+(1-\kappa^2)(1-\rho^2)\rho^{-1}(u_1)_{\rho}=0.
$$

Since $(1-\kappa^2)(1-\rho^2)^2\neq0$, we have
$$
(u_1)_{\rho\rho}+\frac{(1-\kappa^2)(1-\rho^2)}{\rho(1-\kappa^2)(1-\rho^2)^2}(u_1)_{\rho}=0.
$$

Obviously, there is only spatial derivative. Multiplying above equation by $(u_1)_{\rho}$, and then integrating by parts on $[0,1]$, we get
$$
(u_1)_{\rho}^2(\tau,1)+\int_0^1\frac{2(1-\kappa^2)(1-\rho^2)}{\rho(1-\kappa^2)(1-\rho^2)^2}(u_1)_{\rho}^2d\rho=0.
$$
Since $(u_1)_{\rho}^2(\tau,1)\geq0$ and $\frac{(1-\kappa^2)(1-\rho^2)}{\rho(1-\kappa^2)(1-\rho^2)^2}>0$, it holds
$$
(u_1)_{\rho}^2\equiv0,~~(u_1)_{\rho}^2(\tau,1)=0~~and~~ u_1(\tau,1)=0.
$$
Thus we get $u_1\equiv0$. This combining with $u_2=0$ gives that $\mathcal{A}_0$ is injective.

In what follows, we show the operator $\mathcal{A}_0$ is surjective. $\forall \textbf{f}=(f_0,f_1)\in\mathbb{C}_e^{\infty}[0,1]^2$, $\mathcal{A}_0u=\textbf{f}$ implies that
$$
u_2=f_1,
$$
and
$$
\aligned
(1-\kappa^2)(1-\rho^2)^2(u_2)_{\rho\rho}&+(1-\kappa^2)(1-\rho^2)\rho^{-1}(u_2)_{\rho}\\
&=[4\kappa^2-1+(\kappa^2-1)^2\rho^2]f_1\\
&+2(1-\kappa^2)\rho(1-\rho^2)f_1'+(1+(\kappa^2-1)\rho^2)f_0,
\endaligned
$$
which can be seen as an ODE with singular term, so a routine calculation (e.g. see \cite{Lia}) shows that it has an unique solution
$$
u_1=\int_0^{\rho}K(\rho,s)\bar{f}(s)ds,
$$
where
$$
\aligned
K(\rho,s)&=s^{-1}e^{-\int_{s}^1\frac{(1-\kappa^2)(1-s_0^2)}{s_0^{2}(1-\kappa^2)(1-s_0^2)^2}ds_0}\int_s^{\rho}e^{\int_{s_1}^1\frac{(1-\kappa^2)(1-s_0^2)}{s_0^{2}(1-\kappa^2)(1-s_0^2)^2}ds_0}ds_1,\\
\bar{f}(s)&=\frac{4\kappa^2-1+(\kappa^2-1)^2s^2}{(1-\kappa^2)(1-s^2)^2}f_1+\frac{2(1-\kappa^2)s(1-s^2)}{(1-\kappa^2)(1-s^2)^2}f_1'+\frac{1+(\kappa^2-1)s}{(1-\kappa^2)(1-s^2)^2}f_0.
\endaligned
$$

Up until now, we have found $u=(u_1,u_2)\in\mathbb{D}(\mathcal{A}_0)$, which satisfies $\mathcal{A}_0u=\textbf{f}$. Thus the existence of $\mathcal{A}^{-1}$ has been shown.
Furthermore, the Lumer-Phillips Theorem \cite{Pa} gives that $\mathcal{A}_0$ generates a $\mathbb{C}_0$-semigroup in $\mathcal{H}$.
\end{proof}

\begin{lemma}
The operator $\mathcal{A}_1:\mathcal{H}\rightarrow\mathcal{H}$ defined in (\ref{E3-34}) is compact.
\end{lemma}
\begin{proof}
Let $\{u_l\}_{l\in\mathbb{N}}=\{((u_1)_l,(u_2)_l)\}_{l\in\mathbb{N}}$ be a sequence that is uniformly bounded in $\mathcal{H}$. By (\ref{E3-34}), it has
$$
\aligned
\mathcal{A}_1u_l&=\left(
\begin{array}{cccc}
0&0\\
\frac{4\kappa^2}{1+(\kappa^2-1)\rho}+1-(\rho^{-1}a(\rho)+b(\rho))\partial_{\rho}&0
\end{array}
\right)u_l\\
&=(0,\frac{4\kappa^2(u_1)_l}{1+(\kappa^2-1)\rho}+(u_1)_l-(\rho^{-1}a(\rho)+b(\rho))\partial_{\rho}(u_1)_l)^T,
\endaligned
$$
where $a(\rho)$ and $b(\rho)$ are bounded in $(0,1)$ by (\ref{E3-34R1}).
This implies that $\mathcal{A}_1u_l$ is also bounded in $\mathcal{H}$. Moreover, for any two uniformly bounded sequences
$\{u_l\}_{l\in\mathbb{N}}=\{((u_1)_l,(u_2)_l)\}_{l\in\mathbb{N}}$ and $\{u_{l'}\}_{l'\in\mathbb{N}}=\{((u_1)_{l'},(u_2)_{l'})\}_{l'\in\mathbb{N}}$ in $\mathcal{H}$, it holds
$$
\|\mathcal{A}_1u_l-\mathcal{A}_1u_{l'}\|_{\mathcal{H}}\lesssim\|(u_1)_l-(u_1)_{l'}\|_{\mathbb{H}^2(0,1)},
$$
which implies that the sequence $\{\mathcal{A}_1u_l\}_{l\in\mathbb{N}}$ contains Cauchy sequence. This completes the proof.
\end{proof}

By (\ref{E3-35}), Lemma 3.1-3.4 and the bounded perturbation theorem (Theorem 1.3 in p.158 of \cite{Kl}), we can conclude our main result in this subsection.
\begin{proposition}
The operator $\mathcal{A}$ defined in (\ref{E3-32R1}) generates a $\mathbb{C}_0$-semigroup $(\textbf{S}(\tau))_{\tau\geq0}$ in $\mathcal{H}$.
Moreover, the Cauchy problem
\begin{eqnarray*}
&&\frac{d}{d\tau}\psi(\tau)=\mathcal{A}\psi(\tau),~~\tau>0,\\
&&\psi(0)=\psi_0,~~\psi_0\in\mathbb{D}(\mathcal{A}_0),
\end{eqnarray*}
admits a unique solution
\begin{equation*}
\psi(\tau)=\textbf{S}(\tau)\psi_0,~~\forall \tau\geq0,
\end{equation*}
where the initial data $\psi_0=(v_0^T(\rho),v_1^T(\rho))$ is given in (\ref{E2-8R1})-(\ref{E2-8R2}).
\end{proposition}


\subsection{The spectrum of linearized operator at the initial approximation step}

We now carry out the analysis of spectrum for the linear operator in equation (\ref{E2-7}). Assume that constant $\kappa\sim1$ and $\kappa<1$.
Let
$$v(\tau,\rho)=e^{\nu\tau}v_{\nu},$$ 
then equation (\ref{E2-7}) is reduced into a singular ODE
\begin{equation}\label{E3-4}
\begin{array}{lll}
(1-\kappa^2)(1-\rho^2)^2v''_{\nu}&+&\Big[\rho^{-1}(1-\kappa^2)(1-\rho^2)-2\nu(1-\kappa^2)\rho(1-\rho^2)\Big]v'_{\nu}\\
&&-\Big[\nu^2\Big(1+(\kappa^2-1)\rho^2\Big)+\nu\Big(4\kappa^2-1+(\kappa-1)^2\rho^2\Big)-4\kappa^2\Big]v_{\nu}=0,
\end{array}
\end{equation}
where $\nu$ denotes the spectrum of (\ref{E2-7}).

By the Definition 2.1, we should prove that ODE (\ref{E3-4}) has an analytic solution only if $Re~\nu<0$. Obviously, (\ref{E3-4}) can be rewritten as follows
$$
\aligned
v''_{\nu}+\rho^{-1}(1-\rho^2)^{-1}(1-2\rho^2\nu)v'_{\nu}-&(1-\kappa^2)^{-1}(1-\rho^2)^{-2}\Big[\nu^2\Big(1+(\kappa^2-1)\rho^2\Big)\\
&+\nu\Big(4\kappa^2-1+(\kappa-1)^2\rho^2\Big)-4\kappa^2\Big]v_{\nu}=0,
\endaligned
$$
which means that there are two singular points at $\rho=0$ and $\rho=1$. Costin-Donninger-Glogic-Huang \cite{Cos2} proved
the mode stabiliy of self-similar solutions for an energy-supercritical Yang-Mills equation by means of 
quasi-solution expansion method, wihch is a different approach with the method in \cite{Cos0}. After that, the mode stable of Bizo\'{n}-Biernat solution (an explicit self-similar of higher dimensional wave map \cite{BB}) for wave maps in higher dimension ($n\geq4$) was proved in \cite{Cos3}. In their approach, the coefficient of quasi-solution satisfied a recurrence relation, which leads to a difference equation. By some transformations, the characteristic equation of a new difference equation has two completely different eigenvalues which is the key point to apply a theorem of Poincar\'{e} (see, for example, \cite{Bus} or \cite{Ela}). But in our case, since the characteristic equation of difference equation has double roots, the theorem of Poincar\'{e} can't be used. 

Frobenius method \cite{Ger} is a powerful method to deal with the existence of analytic solution to singular ODE. The method of Frobenius tells us that we can find a power series solution at $\rho=0$ of the form
\begin{equation}\label{New1-1}
v_{\nu}=\sum_{n=0}^{\infty}a_n\rho^{n+n_0},~~a_0=1,
\end{equation}
where $n_0$ satisfies the indicial polynomial which is the coefficient of the lowest power of $\rho$ in the infinite series, $a_n$ is the coefficients of $v_{\nu}$ which depends on $\nu$ and $\kappa$.

Inserting this series solution (\ref{New1-1}) into (\ref{E3-4}), we get
the indicial polynomial as follows
\begin{equation*}
(1-\kappa^2)n_0^2=0,
\end{equation*}
which means that Frobenius indice $n_0=0$ (double). Furthermore, there is a recurrence relation of the coefficients $\{a_n\}_{n=0}^{\infty}$ as follows
\bel{E3-6}
a_{n+4}+(-2+p_1(n))a_{n+2}+(1+p_2(n))a_n=0,
\ee
where $a_0=1$ and
\bel{E3-15}
p_1(n)={(15+2\nu)n+3\nu^2+(4\kappa^2-{2(1-\kappa)\over 1+\kappa}+15)\nu+24-4\kappa^2\over n^2+(7+2\nu)n+\nu^2+(8-{1-\kappa\over 1+\kappa})\nu+12},
\ee
\bel{E3-16}
p_2(n)=-{(7+2\nu)n+\nu^2+(8-{1-\kappa\over 1+\kappa})\nu+12\over n^2+(7+2\nu)n+\nu^2+(8-{1-\kappa\over 1+\kappa})\nu+12}.
\ee

\begin{lemma}
The $4$th order difference equation (\ref{E3-6}) has $4$ linearly independent formal solutions. Moreover, those solutions are unbounded as $n\rightarrow+\infty$.
\end{lemma}
The existence of solutions to difference equation (\ref{E3-6}) is directly obtained by a result of Birkhoff and Trjitzinsk \cite{Bir3}. The unbounded property of solutions will be proved in the appendix.

Set
\begin{equation*}
R_n=\frac{a_{n+2}}{a_{n}},
\end{equation*}
then by (\ref{E3-6}), it holds
\begin{equation*}
R_{n+2}=2+p_1(n)-(1+p_2(n))\frac{1}{R_n},
\end{equation*}
let $\lim_{n\rightarrow\infty}R_n=R$, then the characteristic equation is
\begin{equation*}
(R-1)^2=0.
\end{equation*}
Obviously, it has $R=1$ (double), so a theorem of Poincar\'{e} can not apply. This means that we can not follow the method in \cite{Cos1,Cos2,Cos3}.
Thus we have to return to solve difference equation (\ref{E3-6}).

Although there are unbounded solutions of difference equation (\ref{E3-6}), we want to know the asymptotic of unbounded solutions. Furthermore, we want to get the radius of convergence of $\sum_{n=0}^{\infty}a_n\rho^{n+n_0}$ with $n_0=0$ in (\ref{New1-1}).
A general procedure of finding Birkhoff-Trjitzinsk expansions is fairly complicated, but in most cases, a simplified procedure is sufficient. In what follows, we use the Newton polygon method to construct such expansion of solutions of (\ref{E3-6}). Newton polygons provide one technique for the study of the behaviour of the roots to a polynomial over a field.

Let $d_n=\frac{a_{n+1}}{a_n}$. Substituting it into (\ref{E3-6}), we get the relation
\begin{equation}\label{E3-21}
d_{n+3}d_{n+2}d_{n+1}d_n+(-2+p_1(n))d_{n+1}d_n+(1+p_2(n))=0.
\end{equation}
By (\ref{E3-15})-(\ref{E3-16}), (\ref{E3-21}) is equivalent to
\begin{equation}\label{E3-22}
p_3(n)d_{n+3}d_{n+2}d_{n+1}d_n+p_4(n)d_{n+1}d_n+p_5(n)=0,
\end{equation}
where
\begin{eqnarray}\label{E3-23}
&&p_3(n)=(n+4)(n+3+2\nu)+\nu^2-{1-\kappa\over 1+\kappa}\nu,\\
\label{E3-24}
&&p_4(n)=\nu^2+\nu(4\kappa^2-2n-1)-n(2n-1)-4\kappa^2,\\
\label{E3-25}
&&p_5(n)=n^2.
\end{eqnarray}

We introduce the concept of Newton polygon, which is taken from page 380 in \cite{Brie}.

\begin{definition}
Let $\Lambda:=\{(i,deg(a_i))|i=1,2,3,\ldots,m\}$ be the set of point in $\mathbb{R}^2$, and $deg(a_i)$ be the degree of polynomial $a_i$. For each point $p_0$ of $\Lambda$, there is a positive quadrant $p+\mathbb{R}^+\times\mathbb{R}^+$ moved up $p_0$. From the union of all these displaced quadrants, we construct the convex null $\cup_{p\in\Lambda}(p+\mathbb{R}^+\times\mathbb{R}^+)$.
Then the compact polygonal path (all the segments having negative slope) is called the Newton polygon of $\Lambda$.
\end{definition}

In our case, the $m$ in $\Lambda$ should be $3$ by noticing (\ref{E3-22}). Before showing the Newton polygon of $\Lambda$, we should make the transfomation
\begin{equation}\label{E3-30}
d_n=n^{-\bar{x}}\tilde{d}_n,
\end{equation}
where $\bar{x}$ will be determined in the construction of the Newton polygon. Then the difference equation (\ref{E3-22}) is changed into
\begin{equation}\label{E3-22R1}
p_3(n)[(n+3)(n+2)(n+1)n]^{-\bar{x}}\tilde{d}_{n+3}\tilde{d}_{n+2}\tilde{d}_{n+1}\tilde{d}_n+p_4(n)[(n+1)n]^{-\bar{x}}\tilde{d}_{n+1}\tilde{d}_n+p_5(n)=0.
\end{equation}

\begin{lemma}
The difference equation (\ref{E3-22R1}) with $\bar{x}=\frac{1}{4}$ has a solution $\tilde{d}_n=n^{\frac{1}{4}}$. Moreover, $\sigma_p(\mathcal{A})\subset\{\nu\in\mathbb{C}:Re~\nu<0\}$.
\end{lemma}
\begin{proof}
We now divide the proof into two steps to determine the asymptotics of $\tilde{d}_n$ to (\ref{E3-22R1}).

\textbf{Step1.} This step finds the asymptotics of $d_n$ by studying the Newton diagram of the characteristic equation (\ref{E3-22}). It is the same with $\tilde{d}_{n+1}$ and so on.
By $p_3(n)$, $p_4(n)$ and $p_5(n)$ defined in (\ref{E3-23})-(\ref{E3-25}), we know that
\begin{equation*}
\Lambda=\{(2,0),(2(1-\bar{x}),2),(2(1-2\bar{x}),4)\}.
\end{equation*}
A polygon is contructed as the convex null of the set $\Lambda$. Denote the edges of the polygon with respect to which the polygon is on the bottom side by $S$. This means that the equation $\beta\leq a\alpha+b$ determines the half-plane containing the polygon, and the straight line bounding this half plane contains an edge of the polygon.

Let $-\mu^{-1}_0$ be the the slope of the steepest segment of the Newton polygon, and $\mu_1$ be the intercept on the $\alpha$-axis of the line through $(0,\tilde{m})$ with slope $-\mu^{-1}_0$. So there is $\mu_1=\tilde{m}\mu_0$. For the difference equation (\ref{E3-22R1}), there are
\begin{equation*}
-\mu^{-1}_0=\frac{4-2}{2(1-2\bar{x})-2(1-\bar{x})}=-\frac{1}{\bar{x}},
\end{equation*}
and $\tilde{m}=4+\frac{1}{\bar{x}}$ and $\mu_1=4\bar{x}+1$. The line through $(0,\tilde{m})$ with slope $-\mu^{-1}_0$ is
\begin{equation}\label{E3-26}
\alpha+\bar{x}\beta=4\bar{x}+1.
\end{equation}

We find a simple way to solve our problem when constructing the Newton polygon, i.e. let the point $(2,0)$ be in the line (\ref{E3-26}).
Thus we get $\bar{x}=\frac{1}{4}$, and
\begin{equation*}
\Lambda=\{(2,0),(\frac{3}{2},2),(1,4)\}.
\end{equation*}
Then the Newton polygon is only a line through $(0,8)$ with slope $\frac{1}{4}$.

\textbf{Step2.} Set
\begin{equation}\label{E3-31}
\tilde{d}_n=n^{\frac{1}{4}},
\end{equation}
then substituting it into (\ref{E3-22R1}), we have
\begin{equation}\label{E3-27}
p_3(n)+p_4(n)+p_5(n)=0.
\end{equation}

Substituting (\ref{E3-23})-(\ref{E3-25}) into (\ref{E3-27}), it has
\bel{E3-28}
2\nu^2+(4\kappa^2-{1-\kappa\over 1+\kappa}+7)\nu+8n+12-4\kappa^2=0.
\ee

Let
\begin{eqnarray*}
&&P(\nu)=2\nu^2+(4\kappa^2-{1-\kappa\over 1+\kappa}+7)\nu+8n+12-4\kappa^2,\\
&&Q(\nu)=(4\kappa^2-{1-\kappa\over 1+\kappa}+7)\nu.
\end{eqnarray*}
Then direct computation shows that
\begin{eqnarray}\label{E3-29}
\frac{Q(\nu)}{P(\nu)}=\frac{1}{2(4\kappa^2-{1-\kappa\over 1+\kappa}+7)^{-1}+1+(8n+12-4\kappa^2)(4\kappa^2-{1-\kappa\over 1+\kappa}+7)^{-1}\nu^{-1}}.
\end{eqnarray}
Since $\kappa\sim1$ and $\kappa<1$, there are
$$
4\kappa^2-{1-\kappa\over 1+\kappa}+7>0,
$$
and
$$
8n+12-4\kappa^2>0.
$$

Hence by the Routh-Hurwitz criterion (see Theorem A in \cite{Wa}) and (\ref{E3-29}), all the $\nu$ in (\ref{E3-28}) have negative real parts. It follows form (\ref{E3-30}) and (\ref{E3-31}) that $\lim_{n\rightarrow\infty}d_n=1$.
\end{proof}


\begin{proposition}
$\sigma(\mathcal{A})\subset\{\nu\in\mathbb{C}:Re~\nu<0\}$.
\end{proposition}
\begin{proof}
By contradiction, there is a $\nu\geq0$ in $\sigma(\mathcal{A})$. Since $\mathcal{A}_0$ is a dissipative operator proven in Lemma 3.2, $\nu$ is
contained in the resolvent set of $\mathcal{A}_0$, and $\nu-\mathcal{A}=(1-\mathcal{A}'R_{\mathcal{A}_0}(\nu))(\nu-\mathcal{A}_0)$. So $1\in\sigma(\mathcal{A}'R_{\mathcal{A}_0}(\nu))$. By the compactness of $\mathcal{A}'R_{\mathcal{A}_0}(\nu)$, there is $1\in\sigma_p(\mathcal{A}'R_{\mathcal{A}_0}(\nu))$. Furthermore, let $u=\mathcal{A}'R_{\mathcal{A}_0}(\nu)$, we have $(\nu-\mathcal{A})u=0$. Consequently, $0\leq\nu\in\sigma_{p}(\mathcal{A}_0)$. This conflicts with $\sigma_p(\mathcal{A})\subset\{\nu\in\mathbb{C}:Re~\nu<0\}$ in Lemma 3.6.
\end{proof}

\subsection{Decay in time of the general approximation solution}

Let constants $l\geq2$ and $0<\eps\ll 1$, we define
$$
\Bcal_{\eps,l}:=\{v\in\mathcal{C}^{l}_{1}:\quad \|v\|_{\mathcal{C}^{l}_{1}}\leq \eps\}.
$$

We denote the general approximation solution by $v^{(0)}+w$. Assume that  $w\in\Bcal_{\eps,l}$.
Linearizing equation (\ref{E2-7RR}) around $v^{(0)}+w$, then the linearized operator takes the form
\bel{YAE3-1}
\aligned
\Lcal[w] v:=&\Big(1+(\kappa^2-1)\rho^2+a_0(w)\Big)v_{\tau\tau}-\Big((1-\kappa^2)(1-\rho^2)^2+a_1(w)\Big)v_{\rho\rho}\\
&+\Big(4\kappa^2-1+(\kappa-1)^2\rho^2+a_2(w)\Big)v_{\tau}+\Big(2\rho(1-\kappa^2)(1-\rho^2)+a_3(w)\Big)v_{\tau\rho}\\
&\quad +\Big(-(1-\kappa^2)(1-\rho^2)\rho^{-1}+a_4(w)\Big)v_{\rho}+\Big(-4\kappa^2+a_5(w)\Big)v,
\endaligned
\ee
where
\bel{YAE3-1RRR1}
\aligned
a_0(w)&:=-2\kappa\rho(1-\rho^2)^{-{1\over 2}}w_{\rho}+w_{\rho}^2,\\
a_1(w)&:=2\kappa(1-\rho^2)^{{1\over 2}}(w+w_{\tau})-(w+w_{\tau})^2,\\
a_2(w)&:=2\Big(\rho^{-1}(w_{\tau}-w)-w_{\tau\rho}-\kappa(1+2\rho^{-1})(1-\rho^2)^{{1\over2}}\Big)w_{\rho}\\
&\quad\quad+2\Big(\kappa(2-\rho^2)(1-\rho^2)^{-{3\over2}}-w_{\rho\rho}\Big)(w-w_{\tau})+w_{\rho}^2-2\kappa(1-\rho^2)^{{1\over2}}w_{\rho\rho}+2\kappa\rho(1-\rho^2)^{-{1\over2}}w_{\tau\rho},\\
a_3(w)&:=2\kappa\rho(1-\rho^2)^{-{1\over2}}(w_{\tau}-w)-2w_{\rho}(w_{\tau}-w-\kappa(1-\rho^2)^{{1\over 2}}),\\
a_4(w)&:=2\kappa\rho^{-1}(1-\rho^2)^{-{1\over2}}\Big((1+\rho^2)w+5\rho(1-\rho^2)w_{\rho}-w_{\tau}
-\rho^2w_{\tau\tau}\Big)-2w_{\rho}(w_{\tau\tau}+w_{\tau}-2w)\\
&\quad\quad-2w_{\tau\rho}(w_{\tau}-w)+\rho^{-1}(w_{\tau}-w)^2-3\rho^{-1}(1-\rho^2)w_{\rho}^2,\\
a_5(w)&:=2\kappa\rho^{-1}(1-\rho^2)^{-{1\over2}}\Big(\rho^3(1-\rho^2)^{-1}(w-w_{\tau})+(1+\rho^2)w_{\rho}-\rho w_{\rho\rho}-\rho^2w_{\tau\rho}\Big)\\
&\quad\quad-2w_{\rho}^2-2w_{\rho\rho}(w-w_{\tau})+2w_{\tau\rho}w_{\rho}-2\rho^{-1}w_{\rho}(w_{\tau}-w).
\endaligned
\ee

We consider the decay in time of solution for the linear problem
\bel{YAE3-2}
\aligned
&\Lcal[w] v=0,\quad \tau>0,\\
& v(0,\rho)=v_0,\quad v_{\tau}(0,\rho)=v_1,
\endaligned
\ee
with the boundary condition
$$
v(\cdot,0)=v(\cdot,\sigma)=0,\quad v_{\rho}(\cdot,0)=v_{\rho}(\cdot,\sigma)=0.
$$

\begin{lemma}
Assume that  $w\in\Bcal_{\eps,l}$. Then the solution of (\ref{YAE3-2}) satisfies
\bel{YAE3-3R1}
\int_0^{\sigma}\Big(v_{\tau}^2+v_{\rho}^2+v^2\Big)d\rho\lesssim e^{-C_{\kappa,\sigma,\eps}\tau}\int_0^{\sigma}\Big(v_0^2+v_0'^2+v_1^2\Big)d\rho,\quad \forall \tau>0,
\ee
where $C_{\kappa,\sigma,\eps}$ is a positive constant depending on $\kappa$, $\sigma$ and $\eps$.
\end{lemma}
\begin{proof}
Let $\mu_1$ and $\mu_2$ be two positive constants, which will be chosen later.
Multiplying both sides of (\ref{YAE3-1}) by $v_{\tau}-\mu_1v_{\rho}+\mu_2 v$ and integrating over $(0,\delta]$, it holds
\bel{YAE3-3}
\aligned
&{d\over d\tau}\int_0^{\sigma}\Big[{1\over2}\Big(4(\mu_2-1)\kappa^2-\mu_2+\mu_2(\kappa-1)^2\rho^2+\mu_2 a_2(w)+a_5(w)\Big)v^2\\
&+\Big(1+(\kappa^2-1)\rho^2+a_0(w)\Big)\Big(v_{\tau}(\mu_2 v-\mu_1v_{\rho})+{1\over2}v_{\tau}^2\Big)+\mu_2\Big(2\rho(1-\kappa^2)(1-\rho^2)+a_3(w)\Big)v_{\rho}v\\
&+{1\over2}\Big((1-\kappa^2)(1-\rho^2)(1-2\mu_1\rho-\rho^2)+a_1(w)-\mu_1a_3(w)\Big)v_{\rho}^2\Big]d\rho\\
&+\int_0^{\sigma}A_1(\tau,\rho)v_{\tau}^2d\rho+{1\over2}\int_0^{\sigma}A_2(\tau,\rho)v^2d\rho+\int_0^{\sigma}A_3(\tau,\rho)v_{\rho}^2d\rho\\
&=\int_0^{\sigma}A_4(\tau,\rho)v_{\tau}v_{\rho}d\rho+\int_0^{\sigma}A_5(\tau,\rho)v_{\tau}vd\rho+\int_0^{\sigma}A_6(\tau,\rho)v_{\rho}vd\rho,
\endaligned
\ee
where
$$
\aligned
A_1(\tau,\rho)&:=5\kappa^2-2-\mu_2+2\Big(2-\kappa-\kappa^2+{\mu_2(1-\kappa^2)\over 2}\Big)\rho^2+\mu_1(1-\kappa^2)\rho-
{\mu_1\over2}{\del a_0(w)\over\del\rho}+a_2(w)\\
&\quad-\mu_2a_0(w)-{1\over2}{\del a_0(w)\over\del\tau}-{1\over2}{\del a_3(w)\over \del\rho},\\
A_2(\tau,\rho)&:=\mu_2\rho^{-2}(1-\kappa^2)(1-\rho^2)+2\mu_2(1-5\kappa^2)-\mu_2({\del a_2(w)\over\del\tau}+{\del a_4(w)\over\del\rho}-a_5(w))\\
&\quad+\mu_1{\del a_5(w)\over\del\rho}-{\del a_5(w)\over\del\tau},\\
A_3(\tau,\rho)&:=\mu_1\rho^{-1}(1-\kappa^2)(1-\rho^2)(1+2\rho^2)+\mu_2(1-\kappa^2)(1-\rho^2)^2-{1\over2}({\del a_1(w)\over \del \tau}+\mu_1{\del a_1(w)\over \del\rho})\\
&\quad+\mu_1{\del a_3(w)\over \del\tau}+\mu_2a_1(w)-\mu_1a_4(w),\\
A_4(\tau,\rho)&:=\mu_1(4\kappa^2-1)+\mu_1(\kappa-1)^2\rho^2+(1-\kappa^2)(1-\rho^2)\Big(2(2+\mu_2)\rho+\rho^{-1}\Big)+\mu_1{\del a_0(w)\over \del\tau}+{\del a_1(w)\over \del\rho}\\
&\quad+\mu_1a_2(w)-a_4(w)+\mu_2a_3(w),
\endaligned
$$
and
$$
\aligned
A_5(\tau,\rho)&:=\mu_2({\del a_0\over\del\tau}+{\del a_3\over\del\tau}),\\
A_6(\tau,\rho)&:=4\mu_2\rho(1-\kappa^2)(1-\rho^2)-\mu_2{\del a_1\over\del\rho}.
\endaligned
$$

We now estimate each of term in (\ref{YAE3-3}).
Note that $w\in\Bcal_{\eps,l}$. By the de l'H\^{o}pital's rule (\ref{E3-33}), we integrate by parts to derive
\bel{YAE3-4}
\aligned
&\Big|\int_0^{\sigma}\Big(-
{\mu_1\over2}{\del a_0(w)\over\del\rho}+a_2(w)-\mu_2a_0(w)-{1\over2}{\del a_0(w)\over\del\tau}-{1\over2}{\del a_3(w)\over \del\rho}\Big)v_{\tau}^2d\rho\Big|\lesssim \eps C_{\kappa,\sigma}\int_0^{\sigma}v_{\tau}^2d\rho,\\
&\Big|\int_0^{\sigma}\Big(-\mu_2({\del a_2(w)\over\del\tau}+{\del a_4(w)\over\del\rho}-a_5(w))+\mu_1{\del a_5(w)\over\del\rho}-{\del a_5(w)\over\del\tau}\Big)v^2d\rho\Big|\lesssim \eps C_{\kappa,\sigma}\int_0^{\sigma}v^2d\rho,\\
&\Big|\int_0^{\sigma}\Big(-{1\over2}({\del a_1(w)\over \del \tau}+\mu_1{\del a_1(w)\over \del\rho})+\mu_1{\del a_3(w)\over \del\tau}+\mu_2a_1(w)-\mu_1a_4(w)\Big)v_{\rho}^2d\rho\Big|\lesssim  \eps C_{\kappa,\sigma}\int_0^{\sigma}v_{\rho}^2d\rho,
\endaligned
\ee
and
\bel{YAE3-4a1}
\aligned
&\Big|\int_0^{\sigma}\Big(\mu_1{\del a_0(w)\over \del\tau}+{\del a_1(w)\over \del\rho}+\mu_1a_2(w)-a_4(w)+\mu_2a_3(w)\Big)v_{\tau}v_{\rho}d\rho\Big|\lesssim  {\eps C_{\kappa,\sigma}\over2}\int_0^{\sigma}(v_{\tau}^2+v_{\rho}^2)d\rho,\\
&\Big|\int_0^{\sigma}\Big({\del a_0\over\del\tau}+{\del a_3\over\del\tau}\Big)v_{\tau}vd\rho\Big|\lesssim  {\eps C_{\kappa,\sigma}\over2}\int_0^{\sigma}(v_{\tau}^2+v^2)d\rho,\\
&\Big|\int_0^{\sigma}{\del a_1\over\del\rho}v_{\rho}vd\rho\Big|\lesssim  {\eps C_{\kappa,\sigma}\over2}\int_0^{\sigma}(v_{\rho}^2+v^2)d\rho,
\endaligned
\ee
where $C_{\kappa,\sigma}$ is a positive constant depending on $\kappa$ and $\sigma$. 

Furthermore, it follows from (\ref{YAE3-4a1}) that
\bel{YAE3-5}
\aligned
&\Big|\int_0^{\sigma}A_4(\tau,\rho)v_{\tau}v_{\rho}d\rho\Big|\\
&\lesssim{1\over2}\int_0^{\sigma}\Big[\mu_1(4\kappa^2-1)+\mu_1(\kappa-1)^2\rho^2+2\rho(1-\kappa^2)\Big((1-\rho^2)(2+\mu_2)-{1\over2}\Big)-\eps C_{\kappa,\sigma}\Big]\Big(v_{\tau}^2+v_{\rho}^2\Big)d\rho,\\
&\Big|\int_0^{\sigma}A_5(\tau,\rho)v_{\tau}vd\rho\Big|\lesssim{\mu_2\eps C_{\kappa,\sigma}\over2}\int_0^{\sigma}\Big(v_{\tau}^2+v^2\Big)d\rho,\\
&\Big|\int_0^{\sigma}A_6(\tau,\rho)v_{\rho}vd\rho\Big|\lesssim{1\over2}\int_0^{\sigma}\Big[4\mu_2\rho(1-\kappa^2)(1-\rho^2)-\eps\mu_2C_{\kappa,\sigma}\Big]\Big(v_{\rho}^2+v^2\Big)d\rho.
\endaligned
\ee

Thus by (\ref{YAE3-4})-(\ref{YAE3-5}), there exists a positive constant $C_{\kappa,\sigma,\eps,\mu_1,\mu_2}$ (it depends on positive constants $\kappa,\sigma,\eps,\mu_1,\mu_2$) such that
the energy inequality (\ref{YAE3-3}) can be reduced into the following form
\bel{YAE3-6}
\aligned
&{d\over d\tau}\int_0^{\sigma}\Big[{1\over2}\Big(4(\mu_2-1)\kappa^2-\mu_2+\mu_2(\kappa-1)^2\rho^2+\mu_2 a_2(w)+a_5(w)\Big)v^2\\
&+\Big(1+(\kappa^2-1)\rho^2+a_0(w)\Big)\Big(v_{\tau}(\mu_2 v-\mu_1v_{\rho})+{1\over2}v_{\tau}^2\Big)+\mu_2\Big(2\rho(1-\kappa^2)(1-\rho^2)+a_3(w)\Big)v_{\rho}v\\
&+{1\over2}\Big((1-\kappa^2)(1-\rho^2)(1-2\mu_1\rho-\rho^2)+a_1(w)-\mu_1a_3(w)\Big)v_{\rho}^2\Big]d\rho+C_{\kappa,\sigma,\eps,\mu_1,\mu_2}\int_0^{\sigma}(v_{\tau}^2+v_{\rho}^2+v^2)d\rho\lesssim0,
\endaligned
\ee
moreover, let $0<\mu_1<1<\mu_2$, we integrate (\ref{YAE3-6}) over $(0,\tau)$ and use Young's inequality to derive
$$
\int_0^{\sigma}\Big(v_{\tau}^2+v_{\rho}^2+v^2\Big)d\rho+
C_{\kappa,\sigma,\eps}\int_0^{\tau}\int_0^{\sigma}\Big(v_{\tau}^2+v_{\rho}^2+v^2\Big)d\rho d\tau\lesssim\int_0^{\sigma}\Big(v_0^2+v_0'^2+v_1^2\Big)d\rho.
$$

Therefore, by Gronwall's inequality, we obtain
$$
\int_0^{\sigma}\Big(v_{\tau}^2+v_{\rho}^2+v^2\Big)d\rho\lesssim e^{-C_{\kappa,\sigma,\eps}\tau}\int_0^{\sigma}\Big(v_0^2+v_0'^2+v_1^2\Big)d\rho.
$$

\end{proof}

In what follows, we derive the $\HH^l$-estimate for the solution of (\ref{YAE3-2}).
We apply the operator $\del_{\rho}^l$ to both sides of (\ref{YAE3-2}) to get
\bel{YAE3-7}
\aligned
\Big(1+(\kappa^2-1)\rho^2&+a_0(w)\Big)\del_{\tau\tau}\del_{\rho}^lv-\Big((1-\kappa^2)(1-\rho^2)^2+a_1(w)\Big)\del_{\rho}^{l+2}v\\
&+\Big(4\kappa^2-1+(\kappa-1)^2\rho^2+a_2(w)\Big)\del_{\tau}\del^l_{\rho}v+\Big(2\rho(1-\kappa^2)(1-\rho^2)+a_3(w)\Big)\del_{\tau}\del_{\rho}^{l+1}v\\
&\quad +\Big(-(1-\kappa^2)(1-\rho^2)\rho^{-1}+a_4(w)\Big)\del_{\rho}^{l+1}v+\Big(-4\kappa^2+a_5(w)\Big)\del_{\rho}^lv=\textbf{f}_l,
\endaligned
\ee
with the initial data
$$
\del_{\rho}^lv(0,\rho)=v^{(l)}_0,\quad \del_{\rho}^lv_{\tau}(0,\rho)=v^{(l+1)}_1,
$$
and the boundary condition
$$
\del_{\rho}^lv(\cdot,0)=\del_{\rho}^{l}v(\cdot,\sigma)=0,\quad \del_{\rho}^{l+1}v(\cdot,0)=\del_{\rho}^{l+1}v(\cdot,\sigma)=0,
$$
where $2\leq l=l_1+l_2$ with $1\leq l_1\leq l$ and $0\leq l_2\leq l-1$, and
\bel{YAE3-7R1}
\aligned
\textbf{f}_l&:=\sum_{l=l_1+l_2}\del_{\rho}^{l_1}\Big(1+(\kappa^2-1)\rho^2+a_0(w)\Big)\del_{\tau\tau}\del_{\rho}^{l_2}v-\sum_{l=l_1+l_2}\del_{\rho}^{l_1}\Big((1-\kappa^2)(1-\rho^2)^2+a_1(w)\Big)\del_{\rho}^{l_2+2}v\\
&\quad+\sum_{l=l_1+l_2}\del_{\rho}^{l_1}\Big(4\kappa^2-1+(\kappa-1)^2\rho^2+a_2(w)\Big)\del_{\tau}\del_{\rho}^{l_2}v+\sum_{l=l_1+l_2}\del_{\rho}^{l_1}\Big(2\rho(1-\kappa^2)(1-\rho^2)+a_3(w)\Big)\del_{\tau}\del_{\rho}^{l_2+1}v\\
&\quad +\sum_{l=l_1+l_2}\del_{\rho}^{l_1}\Big(-(1-\kappa^2)(1-\rho^2)\rho^{-1}+a_4(w)\Big)\del_{\rho}^{l_2+1}v+\sum_{l=l_1+l_2}\del_{\rho}^{l_1}\Big(-4\kappa^2+a_5(w)\Big)\del_{\rho}^{l_2}v.
\endaligned
\ee

Since the small initial data of equation (\ref{YAE3-2}) can be changed into the zero initial data of it by using a transformation given in Proposition 4.2, we only need to use the zero initial data in each iteration step.

\begin{lemma}
Assume that  $w\in\Bcal_{\eps,l}$. Then there is a positive constant $\sigma$ such that for any $\rho\in(0,\sigma]$, the solution of (\ref{YAE3-2}) satisfying
\bel{YAE3-8R1}
\int_0^{\sigma}\Big((\del_{\tau}\del_{\rho}^lv)^2+(\del_{\rho}^{l+1}v)^2+(\del_{\rho}^{l}v)^2\Big)d\rho\lesssim e^{-C_{\kappa,\eps,\sigma,l}\tau}\int_0^{\sigma}\Big((\del^{l}_{\rho}\del_{\tau}v(0,\rho))^2+(v_1^{(l)})^2+(v_0^{(l)})^2\Big)d\rho,\quad \forall \tau>0.
\ee
where $C_{\kappa,\eps,\sigma,l}$ is a positive constant depending on $\kappa$, $\eps$, $\sigma$ and $l$.
\end{lemma}
\begin{proof}
This proof is based on the induction. The $\HH^1$-estimate has been obtained in Lemma 3.8. We now prove the $\HH^l$-estimates with $l\geq2$. Let constants $0<\mu_1<1<\mu_2$.
Multiplying both sides of (\ref{YAE3-7}) by $\del_{\tau}\del_{\rho}^lv-\mu_1\del_{\rho}^{l+1}v+\mu_2\del_{\rho}^{l}v$, then integrating it over $(0,\sigma]$, it holds
\bel{YAE3-8}
\aligned
&{d\over d\tau}\int_0^{\sigma}\Big[{1\over2}\Big(4(\mu_2-1)\kappa^2-\mu_2+\mu_2(\kappa-1)^2\rho^2+\mu_2 a_2(w)+a_5(w)\Big)(\del_{\rho}^{l}v)^2\\
&\quad+\Big(1+(\kappa^2-1)\rho^2+a_0(w)\Big)\Big(\del_{\tau}\del_{\rho}^{l}v(\mu_2 \del_{\rho}^{l}v-\mu_1\del_{\rho}^{l+1}v)+{1\over2}(\del_{\rho}^{l}\del_{\tau}v)^2\Big)\\
&\quad+\mu_2\Big(2\rho(1-\kappa^2)(1-\rho^2)+a_3(w)\Big)\del_{\rho}^{l+1}v\del_{\rho}^{l}v\\
&\quad+{1\over2}\Big((1-\kappa^2)(1-\rho^2)(1-2\mu_1\rho-\rho^2)+a_1(w)-\mu_1a_3(w)\Big)(\del_{\rho}^{l+1}v)^2\Big]d\rho\\
&\quad+\int_0^{\sigma}A_1(\tau,\rho)(\del_{\rho}^{l}\del_{\tau}v)^2d\rho+{1\over2}\int_0^{\sigma}A_2(\tau,\rho)(\del_{\rho}^{l}v)^2d\rho+\int_0^{\sigma}A_3(\tau,\rho)(\del_{\rho}^{l+1}v)^2d\rho\\
&=\int_0^{\sigma}A_4(\tau,\rho)\del_{\rho}^{l}\del_{\tau}v\del_{\rho}^{l+1}vd\rho+\int_0^{\sigma}A_5(\tau,\rho)\del_{\rho}^{l}\del_{\tau}v\del_{\rho}^{l}vd\rho+\int_0^{\sigma}A_6(\tau,\rho)\del_{\rho}^{l+1}v\del_{\rho}^{l}vd\rho\\
&\quad+\int_0^{\sigma}\Big(\del_{\tau}\del_{\rho}^lv-\mu_1\del_{\rho}^{l+1}v+\mu_2\del_{\rho}^{l}v\Big)\textbf{f}_ld\rho.
\endaligned
\ee

One can see equality (\ref{YAE3-8}) has the same structure with equality (\ref{YAE3-3}) except the last term. So we first estimate the last term $\int_0^{\sigma}\Big(\del_{\tau}\del_{\rho}^lv-\mu_1\del_{\rho}^{l+1}v+\mu_2\del_{\rho}^{l}v\Big)\textbf{f}_ld\rho$. On one hand, by (\ref{YAE3-7R1}),
we integrate by parts to compute
\bel{YAE3-12}
\aligned
&\sum_{l=l_1+l_2}\int_0^{\sigma}\del_{\rho}^{l_1}\Big(1+(\kappa^2-1)\rho^2+a_0(w)\Big)\Big(\del_{\tau\tau}\del_{\rho}^{l_2}v\Big)\Big(\del_{\tau}\del_{\rho}^lv-\mu_1\del_{\rho}^{l+1}v+\mu_2\del_{\rho}^{l}v\Big)d\rho\\
&=\sum_{l=l_1+l_2}{d\over d\tau}\int_0^{\sigma}\del_{\rho}^{l_1}\Big(1+(\kappa^2-1)\rho^2+a_0(w)\Big)\Big(\del_{\tau}\del_{\rho}^{l_2}v\Big)\Big(\del_{\tau}\del_{\rho}^lv-\mu_1\del_{\rho}^{l+1}v+\mu_2\del_{\rho}^{l}v\Big)d\rho\\
&\quad-\sum_{l=l_1+l_2}\int_0^{\sigma}\del_{\tau}\del_{\rho}^{l_1}\Big(1+(\kappa^2-1)\rho^2+a_0(w)\Big)\Big(\del_{\tau}\del_{\rho}^{l_2}v\Big)\Big(\del_{\tau}\del_{\rho}^lv-\mu_1\del_{\rho}^{l+1}v+\mu_2\del_{\rho}^{l}v\Big)d\rho\\
&\quad-\sum_{l=l_1+l_2}\int_0^{\sigma}\del_{\rho}^{l_1}\Big(1+(\kappa^2-1)\rho^2+a_0(w)\Big)\Big(\del_{\tau}\del_{\rho}^{l_2}v\Big)\Big(\del_{\tau\tau}\del_{\rho}^lv-\mu_1\del_{\tau}\del_{\rho}^{l+1}v+\mu_2\del_{\tau}\del_{\rho}^lv\Big)d\rho,
\endaligned
\ee
furthermore, note that $w\in\Bcal_{\eps,l}$, by Young's inequality and Poincar\'{e} inequality, it holds
\bel{YAE3-13}
\aligned
&\sum_{l=l_1+l_2}\Big|\int_0^{\sigma}\del_{\tau}\del_{\rho}^{l_1}\Big(1+(\kappa^2-1)\rho^2+a_0(w)\Big)\Big(\del_{\tau}\del_{\rho}^{l_2}v\Big)\Big(\del_{\tau}\del_{\rho}^lv-\mu_1\del_{\rho}^{l+1}v+\mu_2\del_{\rho}^{l}v\Big)d\rho\Big|\\
&\quad\quad\quad\quad\lesssim\eps C_{\kappa,\sigma}l\int_0^{\sigma}\Big((\del_{\tau}\del_{\rho}^lv)^2+(\del_{\rho}^{l+1}v)^2+(\del_{\rho}^{l}v)^2\Big)d\rho,
\endaligned
\ee
and
\bel{YAE3-13R1}
\aligned
&\sum_{l=l_1+l_2}\Big|\int_0^{\tau}\int_0^{\sigma}\del_{\rho}^{l_1}\Big(1+(\kappa^2-1)\rho^2+a_0(w)\Big)\Big(\del_{\tau}\del_{\rho}^{l_2}v\Big)\Big(\del_{\tau\tau}\del_{\rho}^lv-\mu_1\del_{\tau}\del_{\rho}^{l+1}v+\mu_2\del_{\tau}\del_{\rho}^lv\Big)d\rho d\tau\Big|\\
&\quad\quad\quad \lesssim\eps C_{\kappa,\sigma}l\Big[\int_0^{\tau}\int_0^{\sigma}\Big(\del_{\tau}(\del_{\rho}^{l+1}v)^2+(\del_{\tau}\del_{\rho}^lv)^2\Big)d\rho d\tau+\int_0^{\sigma}(v_1^{(l)})^2d\sigma\Big],
\endaligned
\ee
thus by (\ref{YAE3-12})-(\ref{YAE3-13R1}), it holds
\bel{YAE3-14}
\aligned
&\sum_{l=l_1+l_2}\int_0^{\tau}\int_0^{\sigma}\del_{\rho}^{l_1}\Big(1+(\kappa^2-1)\rho^2+a_0(w)\Big)(\del_{\tau\tau}\del_{\rho}^{l_2}v)(\del_{\tau}\del_{\rho}^lv-\mu_1\del_{\rho}^{l+1}v+\mu_2\del_{\rho}^{l}v)d\rho d\tau\\
&\quad\quad\quad \lesssim\eps C_{\kappa,\sigma}l\Big[\int_0^{\tau}\int_0^{\sigma}\Big((\del_{\tau}\del_{\rho}^lv)^2+(\del_{\rho}^{l+1}v)^2+(\del_{\rho}^{l}v)^2\Big)d\rho d\tau+\int_0^{\sigma}(v_1^{(l)})^2d\sigma\Big].
\endaligned
\ee

On the other hand, we use the de l'H\^{o}pital's rule (\ref{E3-33}) and Young's inequality to compute
\bel{YAE3-14R1}
\aligned
&\sum_{l=l_1+l_2}\Big|\int_0^{\sigma}\del_{\rho}^{l_1}\Big((1-\kappa^2)(1-\rho^2)^2+a_1(w)\Big)\del_{\rho}^{l_2+2}v\Big(\del_{\tau}\del_{\rho}^lv-\mu_1\del_{\rho}^{l+1}v+\mu_2\del_{\rho}^{l}v\Big)d\rho\Big|\\
&\quad\quad\quad\lesssim \eps C_{\kappa,\sigma}l\int_0^{\sigma}\Big((\del_{\tau}\del_{\rho}^lv)^2+(\del_{\rho}^{l+1}v_{\rho})^2+(\del_{\rho}^{l}v_{\rho})^2\Big)d\rho,\\
&\sum_{l=l_1+l_2}\Big|\int_0^{\sigma}\del_{\rho}^{l_1}\Big(4\kappa^2-1+(\kappa-1)^2\rho^2+a_2(w)\Big)\del_{\tau}\del_{\rho}^{l_2}v\Big(\del_{\tau}\del_{\rho}^lv-\mu_1\del_{\rho}^{l+1}v+\mu_2\del_{\rho}^{l}v\Big)d\rho\Big|\\
&\quad\quad\quad\lesssim \eps C_{\kappa,\sigma}l\int_0^{\sigma}\Big((\del_{\tau}\del_{\rho}^lv)^2+(\del_{\rho}^{l+1}v_{\rho})^2+(\del_{\rho}^{l}v_{\rho})^2\Big)d\rho,\\
&\sum_{l=l_1+l_2}\Big|\int_0^{\sigma}\del_{\rho}^{l_1}\Big(2\rho(1-\kappa^2)(1-\rho^2)+a_3(w)\Big)\del_{\tau}\del_{\rho}^{l_2+1}v\Big(\del_{\tau}\del_{\rho}^lv-\mu_1\del_{\rho}^{l+1}v+\mu_2\del_{\rho}^{l}v\Big)d\rho\Big|\\
&\quad\quad\quad\lesssim \eps C_{\kappa,\sigma}l\int_0^{\sigma}\Big((\del_{\tau}\del_{\rho}^lv)^2+(\del_{\rho}^{l+1}v_{\rho})^2+(\del_{\rho}^{l}v_{\rho})^2\Big)d\rho,\\
&\sum_{l=l_1+l_2}\Big|\int_0^{\sigma}\del_{\rho}^{l_1}\Big(-(1-\kappa^2)(1-\rho^2)\rho^{-1}+a_4(w)\Big)\del_{\rho}^{l_2+1}v\Big(\del_{\tau}\del_{\rho}^lv-\mu_1\del_{\rho}^{l+1}v+\mu_2\del_{\rho}^{l}v\Big)d\rho\Big|\\
&\quad\quad\quad\lesssim \eps C_{\kappa,\sigma}l\int_0^{\sigma}\Big((\del_{\tau}\del_{\rho}^lv)^2+(\del_{\rho}^{l+1}v_{\rho})^2+(\del_{\rho}^{l}v_{\rho})^2\Big)d\rho,\\
&\sum_{l=l_1+l_2}\Big|\int_0^{\sigma}\del_{\rho}^{l_1}\Big(-4\kappa^2+a_5(w)\Big)\del_{\rho}^{l_2}v\Big(\del_{\tau}\del_{\rho}^lv-\mu_1\del_{\rho}^{l+1}v+\mu_2\del_{\rho}^{l}v\Big)d\rho\Big|\\
&\quad\quad\quad\lesssim \eps C_{\kappa,\sigma}l\int_0^{\sigma}\Big((\del_{\tau}\del_{\rho}^lv)^2+(\del_{\rho}^{l+1}v_{\rho})^2+(\del_{\rho}^{l}v_{\rho})^2\Big)d\rho,
\endaligned
\ee
thus, using (\ref{YAE3-14})-(\ref{YAE3-14R1}), it holds
\bel{YAE3-15R1}
\aligned
&\int_0^{\tau}\int_0^{\sigma}\Big(\del_{\tau}\del_{\rho}^lv-\mu_1\del_{\rho}^{l+1}v+\mu_2\del_{\rho}^{l}v\Big)\textbf{f}_ld\rho d\tau\\
&\lesssim \eps C_{\kappa,\sigma}l\Big[\int_0^{\tau}\int_0^{\sigma}\Big((\del_{\tau}\del_{\rho}^lv)^2+(\del_{\rho}^{l+1}v)^2+(\del_{\rho}^{l}v)^2\Big)d\rho d\tau+\int_0^{\sigma}(v_1^{(l)})^2d\sigma\Big].
\endaligned
\ee

Hence, by (\ref{YAE3-15R1}) and $\HH^{l}\subset\HH^{l-1}$,
we can use the similar method of getting (\ref{YAE3-6}) to derive
\begin{eqnarray}
\label{YAE3-15RRRRR0}
&&{d\over d\tau}\int_0^{\sigma}\Big[{1\over2}\Big(4(\mu_2-1)\kappa^2-\mu_2+\mu_2(\kappa-1)^2\rho^2+\mu_2 a_2(w)+a_5(w)\Big)(\del_{\rho}^lv)^2\\
&&+\Big(1+(\kappa^2-1)\rho^2+a_0(w)\Big)\Big(\del_{\rho}^l\del_{\tau}v(\mu_2 \del_{\rho}^lv-\mu_1\del_{\rho}^{l+1}v)+{1\over2}(\del_{\rho}^l\del_{\tau}v)^2\Big)\\
&&+\mu_2\Big(2\rho(1-\kappa^2)(1-\rho^2)+a_3(w)\Big)\del_{\rho}^{l+1}v\del_{\rho}^lv\\
&&+{1\over2}\Big((1-\kappa^2)(1-\rho^2)(1-2\mu_1\rho-\rho^2)+a_1(w)-\mu_1a_3(w)\Big)(\del_{\rho}^{l+1}v)^2\Big]d\rho\\
&&+C_{\kappa,\sigma,\eps,\mu_1,\mu_2}\int_0^{\sigma}\Big((\del_{\rho}^l\del_{\tau}v)^2+(\del_{\rho}^{l+1}v)^2+(\del_{\rho}^lv)^2\Big)d\rho\lesssim \int_0^{\sigma}(v_1^{(l)})^2d\sigma,
\end{eqnarray}
furthermore, for the fixed constants $\mu_1,\mu_2$, integrating (\ref{YAE3-15RRRRR0}) over $(0,\tau)$, then we use Young's inequality to derive
\bel{YAE3-15}
\aligned
\int_0^{\sigma}\Big((\del_{\tau}\del_{\rho}^lv)^2&+(\del_{\rho}^{l+1}v)^2+(\del_{\rho}^{l}v)^2\Big)d\rho+C_{\kappa,\eps,\sigma,l}\int_0^{\tau}\int_0^{\sigma}\Big((\del_{\tau}\del_{\rho}^lv)^2+(\del_{\rho}^{l+1}v)^2+(\del_{\rho}^{l}v)^2\Big)d\rho d\tau\\
&\lesssim C_{\kappa,\eps,\sigma,l}\int_0^{\sigma}\Big((\del^{l}_{\rho}\del_{\tau}v(0,\rho))^2+(v_1^{(l)})^2+v_0^{(l)}\Big)d\rho,
\endaligned
\ee
where $C_{\kappa,\eps,\sigma,l}$ is a positive constant depending on $\kappa$, $\eps$, $\sigma$ and $l$.

Therefore, by (\ref{YAE3-15}), we apply Gronwall's inequality to obtain
$$
\int_0^{\sigma}\Big((\del_{\tau}\del_{\rho}^lv)^2+(\del_{\rho}^{l+1}v)^2+(\del_{\rho}^{l}v)^2\Big)d\rho\lesssim e^{-C_{\kappa,\eps,\sigma,l}\tau}\int_0^{\sigma}\Big((\del^{l}_{\rho}\del_{\tau}v(0,\rho))^2+(v_1^{(l)})^2+v_0^{(l)}\Big)d\rho.
$$

\end{proof}

We now consider the linear problem (\ref{YAE3-2}) with an external force as follows
\bel{YAE3-16}
\aligned
&\Lcal[w] v=f(\tau,\rho),\quad \tau>0,\\
& v(0,\rho)=v_0,\quad v_{\tau}(0,\rho)=v_1,
\endaligned
\ee
with the boundary condition
$$
v(\cdot,0)=v(\cdot,\sigma)=0,\quad v_{\rho}(\cdot,0)=v_{\rho}(\cdot,\sigma)=0.
$$

Similar to (\ref{YAE3-3R1}) in Lemma 3.7 and (\ref{YAE3-8R1}) in Lemma 3.8, we conclude the following result.

\begin{lemma}
Assume that  $w\in\Bcal_{\eps,l}$ and $f(\tau,\rho)\in\HH^l((0,\sigma])$. Then there is a postive constant $\sigma$ such that for any $\rho\in(0,\sigma)$, the solution of (\ref{YAE3-16}) satisfying
\bel{YAE3-16R1}
\int_0^{\sigma}\Big((\del_{\tau}\del_{\rho}^lv)^2+(\del_{\rho}^{l+1}v)^2+(\del_{\rho}^{l}v)^2\Big)d\rho\lesssim e^{-C_{\kappa,\eps,\sigma,l}\tau}\int_0^{\sigma}\Big((\del^{l}_{\rho}\del_{\tau}v(0,\rho))^2+(v_1^{(l)})^2+(v_0^{(l)})^2+(\del_{\rho}^lf)^2\Big)d\rho,~\forall \tau>0,
\ee
where $C_{\kappa,\eps,\sigma,l}$ is a positive constant depending on $\kappa$, $\eps$, $\sigma$ and $l$.
\end{lemma}

Furthermore, we derive the existence of result on the problem (\ref{YAE3-16}).

\begin{proposition}
Assume that  $w\in\Bcal_{\eps,l}$ and $f(\tau,\rho)\in\CC^2((0,\infty);\HH^l(\Omega))$. Then 
equation (\ref{YAE3-16}) admits a unique solution 
$$
v(\tau,\rho)\in\Ccal^l_{1}:=\bigcap_{i= 0}^1\CC^i((0,\infty);\HH^{l-i}(\Omega)).
$$
 Moreover, there is
\bel{YAE3-17}
\|v(t,x)\|_{\Ccal^l_{1}}\leq\|(v_0,v_1)\|_{\HH^l\times\HH^{l-1}}+\|f(\tau,\rho)\|_{\Ccal^l_{1}}.
\ee

\end{proposition}
\begin{proof}
We first prove the local existence of solution for (\ref{YAE3-16}), then using the decay in time of solution given in Lemma 3.9, the local solution can be extended into the global solution of (\ref{YAE3-16}).
Since $\Big((1-\kappa^2)(1-\rho^2)^2+a_1(w)\Big)>0$ in (\ref{YAE3-1}), the linearized problem (\ref{YAE3-16}) is a strictly hyperbolic linear equation. Thus we can take a standard fixed point iteration process. Let $h= (v,v_{\tau})$. Then linearized equation (\ref{YAE3-16}) can be rewritten as
$$
\del_th+\mathcal{A}(\tau,\rho)h=F(\tau,\rho),
$$
where $F(\tau,\rho)= (0,f(\tau,\rho))$ and the matrix $\mathcal{A}(\tau,\rho)$ is 
\begin{eqnarray*}
\mathcal{A}(\tau,\rho):=\left(
\begin{array}{ccc}
0&1\\
-A_{1}\del_{\rho\rho}+A_2\del_{\rho}+A_3& A_4\del_{\rho}+A_5
\end{array}
\right),
\end{eqnarray*}
and the coefficients
$$
\aligned
&A_1:=\Big(1+(\kappa^2-1)\rho^2+a_0(w)\Big)^{-1}\Big((1-\kappa^2)(1-\rho^2)^2+a_1(w)\Big),\\
&A_2:=\Big(1+(\kappa^2-1)\rho^2+a_0(w)\Big)^{-1}\Big(-(1-\kappa^2)(1-\rho^2)\rho^{-1}+a_4(w)\Big),\\
&A_3:=\Big(1+(\kappa^2-1)\rho^2+a_0(w)\Big)^{-1}\Big(-4\kappa^2+a_5(w)\Big),\\
&A_4:=\Big(1+(\kappa^2-1)\rho^2+a_0(w)\Big)^{-1}\Big(2\rho(1-\kappa^2)(1-\rho^2)+a_3(w)\Big),\\
&A_5:=\Big(1+(\kappa^2-1)\rho^2+a_0(w)\Big)^{-1}\Big(4\kappa^2-1+(\kappa-1)^2\rho^2+a_2(w)\Big).
\endaligned
$$

Note $w\in\Bcal_{\eps,l}$. Following \cite{Sog}, by
the standard fixed point iteration and a priori estimate (\ref{YAE3-16R1}) in Lemma 3.9, we obtain the approximation problem
$$
h^{(m)}=h_0-\int_0^{\tau}\Big(\mathcal{A}(s,\rho)h^{(m-1)}+F(s,\rho)\Big)d\tau
$$
has a Cauchy sequence $\{h^{(m)}\}_{m\in\mathbb{Z}^+}$ in $\mathcal{C}^l_{1}$, whose limit is $h(\tau,\rho)$ which solves the linearized equation (\ref{YAE3-16}) in $(0,T]$. Furhermore, by the decay in time estimate (\ref{YAE3-16R1}), the local solution $h(\tau,\rho)$ can be extended into a global solution of (\ref{YAE3-16}). 
The estimate (\ref{YAE3-17}) is directly from the estimate (\ref{YAE3-16R1}).

\end{proof}


\section{Nonlinear stability of explicit lightlike self-similar solutions }\setcounter{equation}{0}

\subsection{The approximation scheme}
In this section, we will construct a solution of nonlinear equation (\ref{E2-7RR}) by using the Nash-Moser iteration scheme, which has been used in \cite{Yan,Yan1,YZ,ZY}.
Recall that we have chosen the initial approximation function as follows
$$
v^{(0)}(\tau,\rho)=(\kappa-1)\phi(\rho),
$$
where $\phi(\rho)$ is defined in (\ref{E2-2}).

We set
$$
v(\tau,\rho)=v^{(0)}(\tau,\rho)+w(\tau,\rho),
$$
where $w(\tau,\rho)$ satisfies the following non-autonomous nonlinear equation
\bel{YAE4-1}
\aligned
\Big(1+(\kappa^2-1)\rho^2\Big)w_{\tau\tau}&-(1-\kappa^2)(1-\rho^2)^2w_{\rho\rho}+\Big(4\kappa^2-1+(\kappa-1)^2\rho^2\Big)w_{\tau}\\
&\quad +2\rho(1-\kappa^2)(1-\rho^2)w_{\tau\rho}-(1-\kappa^2)(1-\rho^2)\rho^{-1}w_{\rho}-4\kappa^2w=f(\rho,w),
\endaligned
\ee
where
$$
\aligned
f(\rho,w):=&2\kappa(1-\rho^2)^{\frac{3}{2}} w_{\rho}^2-(1-\rho^2)w_{\rho}(w_{\tau\tau}+w_{\tau}-2w)\Big(w_{\rho}-2\kappa\rho(1-\rho^2)^{-\frac{1}{2}}\Big)\\
&\quad +\kappa(1-\rho^2)^{-\frac{1}{2}}(w-w_{\tau})^2-(1-\rho^2)w_{\rho\rho}\Big[(w-w_{\tau})^2+2\kappa(1-\rho^2)^{\frac{1}{2}}(w-w_{\tau})\Big]\\
&\quad -2\kappa\rho(1-\rho^2)^{\frac{1}{2}}w_{\tau\rho}(-w+w_{\tau})-2\kappa(1-\rho^2)^{\frac{3}{2}}w_{\rho}w_{\rho\tau}\\
&\quad +2(1-\rho^2)w_{\rho}w_{\rho\tau}(-w+w_{\tau})-(1-\rho^2)\rho^{-1}w_{\rho}\Big[(w-w_{\tau})^2+2\kappa(1-\rho^2)^{\frac{1}{2}}(w-w_{\tau})\Big]\\
&\quad +\kappa(1-\rho^2)^{\frac{1}{2}}(w-w_{\tau})^2-(\rho-\rho^{-1})(1-\rho^2)\Big(w_{\rho}^3-3\kappa\rho(1-\rho^2)^{-\frac{1}{2}}w_{\rho}^2\Big)\\
&\quad-2\kappa(1-\kappa^2)(1-\rho^2)^{-\frac{1}{2}}.
\endaligned
$$

Here we supplement equation (\ref{YAE4-1}) with the zero initial data for convenience of computation, i.e.
\bel{YAE4-1RRR1}
w(0,\rho)=0,\quad w_{\tau}(0,\rho)=0,
\ee
and the boundary condition
\bel{YAE4-1RRR2}
w(\cdot,0)=w(\cdot,\sigma)=0,\quad w_{\rho}(\cdot,0)=w_{\rho}(\cdot,\sigma)=0.
\ee

Since $\kappa\sim1$ and $\rho\in(0,\sigma]$, the non-autonomous term has the property
\bel{YAE4-2R1}
2\kappa(1-\kappa^2)(1-\rho^2)^{-\frac{1}{2}}\sim\eps_0\ll1.
\ee

We introduce a family of smooth operators possessing the following properties.
\begin{lemma}\cite{Alin,H1}
There is a family $\{\Pi_{\theta}\}_{\theta\geq1}$ of smoothing operators in the space $\mathbb{H}^{k}((0,\sigma])$ acting on the class of functions such that
\bel{E4-0}
\aligned
&\|\Pi_{\theta}w\|_{\HH_{k_1}}\leq C\theta^{(k_1-k_2)_+}\|w\|_{\HH^{k_2}},~~k_1,~k_2\geq0,\\
&\|\Pi_{\theta}w-w\|_{\HH^{k_1}}\leq C\theta^{k_1-k_2}\|w\|_{\HH^{k_2}},~~0\leq k_1\leq k_2,\\
&\|\frac{d}{d\theta}\Pi_{\theta}w\|_{\HH^{k_1}}\leq C\theta^{(k_1-k_2)_+-1}\|w\|_{\HH^{k_2}},~~k_1,~k_2\geq0,
\endaligned
\ee
where $C$ is a positive constant and $(k_1-k_2)_+:=\max(0,k_1-k_2)$. 
\end{lemma}

In our iteration scheme, we set
$$
\theta=N_m=2^m,\quad \forall m= 0,1,2,\ldots,
$$
then by (\ref{E4-0}), it holds
\bel{E4-1}
\|\Pi_{N_m}v\|_{\HH^{k_1}}\lesssim N_m^{k_1-k_2}\|v\|_{\HH^{k_2}},\quad \forall k_1\geq k_2.
\ee

We consider the approximation problem of nonlinear equation (\ref{YAE4-1}) as follows
\bel{E4-3}
\aligned
\Jcal(w):=
\Big(1+(\kappa^2&-1)\rho^2\Big)w_{\tau\tau}-(1-\kappa^2)(1-\rho^2)^2w_{\rho\rho}+\Big(4\kappa^2-1+(\kappa-1)^2\rho^2\Big)w_{\tau}\\
&\quad +2\rho(1-\kappa^2)(1-\rho^2)w_{\tau\rho}-(1-\kappa^2)(1-\rho^2)\rho^{-1}w_{\rho}-4\kappa^2w-\Pi_{N_m}f(\rho,w).
\endaligned
\ee
We denote the $m$-th approximation solution of (\ref{E4-3}) by $w^{(m)}$. Then the error step $h^{(m)}$ is given by
$$
h^{(m)}:=w^{(m)}-w^{(m-1)},\quad for \quad m=1,2,\ldots,
$$
by which, we get
$$
w^{(m)}=w^{(0)}+\sum_{i=1}^mh^{(i)}.
$$
Our target is to prove that nonlinear equation (\ref{YAE4-1}) admits a global solution. It is equivalent to show the series $\sum_{i=1}^mh^{(i)}$ is convergence.

Linearizing nonlinear equation (\ref{YAE4-1}) around $h^{(m)}$, we get the linearized operator
$$
\aligned
&\Lcal[w^{(m-1)}](h^{(m)}):=\Big(1+(\kappa^2-1)\rho^2+a_0(w^{(m-1)})\Big)h^{(m)}_{\tau\tau}-\Big((1-\kappa^2)(1-\rho^2)^2+a_1(w^{(m-1)})\Big)h^{(m)}_{\rho\rho}\\
&\quad\quad\quad+\Big(4\kappa^2-1+(\kappa-1)^2\rho^2+a_2(w^{(m-1)})\Big)h^{(m)}_{\tau}+\Big(2\rho(1-\kappa^2)(1-\rho^2)+a_3(w^{(m-1)})\Big)h^{(m)}_{\tau\rho}\\
&\quad\quad\quad +\Big(-(1-\kappa^2)(1-\rho^2)\rho^{-1}+a_4(w^{(m-1)})\Big)h^{(m)}_{\rho}+\Big(-4\kappa^2+a_5(w^{(m-1)})\Big)h^{(m)},
\endaligned
$$
where the coefficients $a_0(w^{(m-1)})$, $a_1(w^{(m-1)})$, $a_2(w^{(m-1)})$, $a_3(w^{(m-1)})$, $a_4(w^{(m-1)})$ and $a_5(w^{(m-1)})$ are given in (\ref{YAE3-1RRR1}), but instead of $w$ and its derivatives on $\tau$ and $\rho$ by $w^{(m-1)}$ and its derivatives on $\tau$ and $\rho$ in those coefficients, respectively.

Let constant $l\geq2$.
We choose the approximation function $w^{(0)}\in\Ccal_1^l$ such that the error term
$$
E^{(0)}:=\Jcal(w^{(0)})
$$
satisfies
\bel{E4-5}
\aligned
&w^{(0)}\neq0,\quad and \quad\|w^{(0)}\|_{\Ccal_1^{l}}\lesssim \eps_0<\eps,\\
&\|E^{(0)}\|_{\Ccal_1^{l}}\lesssim\eps_0<\eps,\\
&w^{(0)}(\cdot,\rho)|_{\rho\in\del\Omega}=w_{\rho}^{(0)}(\cdot,\rho)|_{\rho\in\del\Omega}=0.
\endaligned
\ee

Since the non-autonomous term satisfies (\ref{YAE4-2R1}), 
it is easy to check that (\ref{E4-5}) holds for a sufficient small $\eps_0$.

The $m$-th error terms is defined by
\bel{E4-6}
R(h^{(m)}):=\Jcal(w^{(m-1)}+h^{(m)})-\Jcal(w^{(m-1)})-\Lcal[w^{(m-1)}](h^{(m)}),
\ee
which is also the nonlinear term in approximation equation (\ref{E4-3}) at $w^{(m-1)}$. The exact form of nonlinear term (\ref{E4-6}) is very complicated, here we does not write it down. 

\begin{lemma}
Let $l\geq2$.
Assume that $w^{(m-1)}\in\Bcal_{\eps,l}$. Then it holds
\bel{E4-9R1}
\|R(h^{(m)})\|_{\Ccal_1^l}\lesssim N_m^4\|h^{(m)}\|^2_{\Ccal_1^{l}},\quad \forall \tau\in\RR^+.
\ee
\end{lemma}
\begin{proof}
We notice that the highest order of nonlinear term in (\ref{YAE4-1}) is $3$, and the highest order of derivatives on $\rho$ and $\tau$ in (\ref{E4-6})
is $2$. Since the solution of (\ref{E4-3}) should be constructed in $\Bcal_{\eps,l}$, it holds
$$
\|h^{(m)}\|_{\Ccal_1^{l}}\leq \eps\ll1,\quad\forall m\in\NN,
$$ 
which means that
$$
\|h^{(m)}\|^p_{\Ccal_1^{l}}\leq \|h^{(m)}\|^2_{\Ccal_1^{l}},\quad for\quad p\geq2.
$$

Applying (\ref{E4-1}) and Young's inequality to estimate each term in $R(h^{(m)})$, we obtain 
$$
\|R(h^{(m)})\|_{\Ccal_1^l}\lesssim \|h^{(m)}\|^2_{\Ccal_1^{l+2}}\lesssim N_m^4\|h^{(m)}\|^2_{\Ccal_1^{l}}.
$$

\end{proof}

The following Lemma is to construct the $m$-th approximation solution.

\begin{lemma}
Let $l\geq2$.
Assume that $w^{(m-1)}\in\Bcal_{\eps,l}$.
The linear problem
$$
\aligned
&\Lcal[w^{(m-1)}](h^{(m)})=E^{(m-1)},\\
&h^{(m)}(0,\rho)=0,\quad h^{(m)}_{\tau}(0,\rho)=0,
\endaligned
$$
with the boundary condition 
$$
h^{(m)}(\cdot,0)=h^{(m)}(\cdot,\sigma)=0,\quad h_{\rho}^{(m)}(\cdot,0)=h_{\rho}^{(m)}(\cdot,\sigma)=0,
$$
admits a solution $h^{(m)}\in\Ccal_1^{l}$ satisfying
\bel{E4-10R1}
\|h^{(m)}\|_{\Ccal_1^{l}}\lesssim \|E^{(m-1)}\|_{\Ccal_1^{l}},
\ee
where the error term
\bel{E4-10R2}
E^{(m-1)}:=\Jcal(w^{(m-1)})=R(h^{(m-1)}).
\ee
\end{lemma}
\begin{proof}
Assume that $w^{(0)}$ satisfying (\ref{E4-5}). The $m-1$-th approximation solution is 
$$
w^{(m-1)}=w^{(0)}+\sum_{i=1}^{m-1}h^{(i)}.
$$
Then we will find the $m$-th approximation solution $w^{(m)}$, which is equivalent to find $h^{(m)}$ such that 
\bel{E4-7}
w^{(m)}=w^{(m-1)}+h^{(m)}.
\ee
Substituting (\ref{E4-7}) into (\ref{E4-3}), it holds
$$
\Jcal(w^{(m)})=\Jcal(w^{(m-1)})+\Lcal[w^{(m-1)}](h^{(m)})+R(h^{(m)}).
$$

Let 
$$
\Jcal(w^{(m-1)})+\Lcal[w^{(m-1)}](h^{(m)})=0,
$$
we supplement it with the zero initial data
$$
h^{(m)}(0,\rho)=0,\quad h^{(m)}_{\tau}(0,\rho)=0,
$$
and the boundary condition 
$$
h^{(m)}(\cdot,0)=h^{(m)}(\cdot,\sigma)=0,\quad h_{\rho}^{(m)}(\cdot,0)=h_{\rho}^{(m)}(\cdot,\sigma)=0.
$$

By Proposition 3.3, the zero initial data problem admits a solution $h^{(m)}\in\Ccal_1^{l}$. Furthermore, by (\ref{YAE3-17}), it satisfies
$$
\|h^{(m)}\|_{\Ccal_1^{l}}\lesssim \|\Jcal(w^{(m-1)})\|_{\Ccal_1^{l}}.
$$

Moreover, one can see the $m$-th error term $E^{(m)}$ such that
$$
E^{(m)}:=\Jcal(w^{(m)})=R(h^{(m)}).
$$
\end{proof}

\subsection{Convergence of the approximation scheme}

For some fixed $l>2$, let $2\leq\bar{k}<k_0\leq k\leq l$ and
$$
\aligned
&k_m:=\bar{k}+\frac{k-\bar{k}}{2^m},\\
&\alpha_{m+1}:=k_m-k_{m+1}=\frac{k-\bar{k}}{2^{m+1}},
\endaligned
$$
which gives that
\bel{EX1-1}
k_0>k_1>\ldots>k_m>k_{m+1}>\ldots.
\ee

\begin{proposition}
Equation (\ref{YAE4-1}) with the initial data (\ref{YAE4-1RRR1}) and boundary condition (\ref{YAE4-1RRR2}) admits a global solution 
$$
w^{(\infty)}(\tau,\rho)=w^{(0)}(\tau,\rho)+\sum_{m=1}^{\infty}h^{(m)}(\tau,\rho)\in\Ccal_1^{2},
$$
where $w^{(0)}\in\HH^l$ satisfies (\ref{E4-5}).

Moreover, it holds
$$
\|w^{(\infty)}(\tau,\rho)\|_{\Ccal_1^2}\lesssim\eps.
$$
\end{proposition}
\begin{proof}
The proof is based on the induction. Note that $N_m=N_0^m$ with $N_0>1$. $\forall m=1,2,\ldots$, we claim that there exists a sufficient small positive constant $d$ such that
\bel{E4-12}
\aligned
&\|h^{(m)}\|_{\Ccal_1^{k_m}}<d^{2^m},\\
&\|E^{(m-1)}\|_{\Ccal_1^{k_m}}<d^{2^{m+1}},\\
&w^{(m)}\in\Bcal_{\eps,k_m}.
\endaligned
\ee

For the case of $m=1$, we recall that the assumption (\ref{E4-5}) on $w^{(0)}$, i.e. 
$$
\aligned
&w^{(0)}\neq0,\quad and \quad\|w^{(0)}\|_{\Ccal_1^{l}}\lesssim \eps_0,\\
&\|E^{(0)}\|_{\Ccal_1^{l}}\lesssim\eps_0.
\endaligned
$$
So by (\ref{E4-10R1}), let $0<\eps_0<N_0^{-8}d^2<{\eps\over2}\ll1$, it holds
$$
\|h^{(1)}\|_{\Ccal_1^{k_1}}\lesssim \|E^{(0)}\|_{\Ccal_1^{k_0}}\lesssim \eps_0<d.
$$
Moreover, by (\ref{E4-9R1}) and (\ref{E4-10R2}), we derive
$$
\|E^{(1)}\|_{\Ccal_1^{k_1}}\lesssim\|R_1(h^{(1)})\|_{\Ccal_1^{k_1}}\lesssim N_1^4\|h^{(1)}\|^2_{\Ccal_1^{k_1}}\lesssim \eps_0 N_1^4<d^2,
$$
and 
$$
\|w^{(1)}\|_{\Ccal_1^{k_1}}\lesssim\|w^{(0)}\|_{\Ccal_1^{k_1}}+\|h^{(1)}\|_{\Ccal_1^{k_1}}\lesssim 2\eps_0<\eps.
$$
which means that $w^{(1)}\in\Bcal_{\eps,k_1}$.

Assume that the case of $m-1$ holds, i.e.
\bel{E4-13}
\aligned
&\|h^{(m-1)}\|_{\Ccal_1^{k_{m-1}}}<d^{2^{m-1}},\\
&\|E^{(m-1)}\|_{\Ccal_1^{k_{m-1}}}<d^{2^{m}},\\
&w^{(m-1)}\in\Bcal_{\eps,k_{m-1}},
\endaligned
\ee
then we prove the case of $m$ holds. Using (\ref{E4-10R1}) and (\ref{E4-13}), it holds
\bel{E4-14R1}
\|h^{(m)}\|_{\Ccal_1^{k_m}}\lesssim \|E^{(m-1)}\|_{\Ccal_1^{k_{m}}}< \|E^{(m-1)}\|_{\Ccal_1^{k_{m-1}}}<d^{2^{m}},
\ee
which combining with (\ref{E4-9R1}), (\ref{E4-10R2}) and (\ref{EX1-1}), it holds
\bel{E4-14}
\aligned
\|E^{(m)}\|_{\Ccal_1^{k_{m}}}&=\|R(h^{(m)})\|_{\Ccal_1^{k_{m}}}\\
&\lesssim N_{m-1}^4\|E^{(m-1)}\|^2_{\Ccal_1^{k_{m-1}}}\\
&\lesssim N_0^{4(m-1)+8(m-2)}\|E^{(m-2)}\|^{2^2}_{\Ccal_1^{k_{m-2}}}\\
&\lesssim \ldots,\\
&\lesssim (N_0^8\|E_0\|_{\Ccal_1^{k_0}})^{2^m}.
\endaligned
\ee

So by (\ref{E4-5}),  there is a sufficient small positive constant $\eps_0$ such that
$$
0<N_0^8\|E_0\|_{\Ccal_1^{k_0}}<N_0^8\eps_0<d^2,
$$
which combining with (\ref{E4-14}) gives that 
$$
\|E^{(m)}\|_{\Ccal_1^{k_{m}}}<d^{2^{m+1}}.
$$
On the other hand, by (\ref{E4-14R1}), it holds
$$
\|w^{(m)}\|_{\Ccal_1^{k_{m}}}\lesssim \|w^{(m-1)}\|_{\Ccal_1^{k_{m-1}}}+\|h^{(m)}\|_{\Ccal_1^{k_{m}}}\lesssim \sum_{i=1}^md^{2^i}+\eps_0<\eps.
$$
This means that $w^{(m)}\in\Bcal_{\eps,k_m}$. Hence we conclude that (\ref{E4-12}) holds.

Furthermore, it follows from (\ref{E4-12}) that the error term goes to $0$ as $m\rightarrow\infty$, i.e.
$$
\lim_{m\rightarrow\infty}E^{(m)}=0.
$$ 

Therefore, equation (\ref{YAE4-1}) with the initial data (\ref{YAE4-1RRR1}) and boundary condition (\ref{YAE4-1RRR2}) admits a solution 
$$
w^{(\infty)}=w^{(0)}+\sum_{m=1}^{\infty}h^{(m)}\in\Ccal_1^{2}.
$$
\end{proof}

Let $w_0,w_1\in\HH^l$ with $l>2$. We supplement equation (\ref{YAE4-1}) with small initial data 
\bel{Re1-1}
w(0,\rho)=\eps w_0(\rho),\quad w_{\tau}(0,\rho)=\eps w_1(\rho),
\ee
where 
$$
\aligned
&w_0(\rho)|_{\rho\in\del\Omega}=w_1(\rho)|_{\rho\in\del\Omega}=0,\\
&w'_0(\rho)|_{\rho\in\del\Omega}=w'_1(\rho)|_{\rho\in\del\Omega}=0.
\endaligned
$$

Then we have the following result.

\begin{proposition}
Equation (\ref{YAE4-1}) with the initial data (\ref{Re1-1})
and boundary condition (\ref{YAE4-1RRR2}) admits a global solution 
$$
w^{(\infty)}(\tau,\rho)=w^{(0)}(\tau,\rho)+\sum_{m=1}^{\infty}h^{(m)}(\tau,\rho)+\eps w_0(\rho)-\eps (e^{-\tau}-1)w_1(\rho),
$$
where $\sum_{m=1}^{\infty}h^{(m)}(\tau,\rho)\in\Ccal_1^{2}$.

Moreover, it holds
$$
\|w^{(\infty)}(\tau,\rho)\|_{\Ccal_1^2}\lesssim\eps.
$$
\end{proposition}
\begin{proof}

We introduce an auxiliary function
$$
\overline{w}(\tau,\rho)=w(\tau,\rho)-\eps w_0(\rho)+\eps (e^{-\tau}-1)w_1(\rho),
$$
then the small initial data (\ref{Re1-1}) is reduced into
$$
\overline{w}(0,\rho)=0,\quad \overline{w}_{\tau}(0,\rho)=0,
$$
and equation (\ref{YAE4-1}) is transformed into an equation of $\overline{w}(\tau,\rho)$. This new equation is only added by some more terms on $\eps w_0$ and $\eps (e^{-\tau}-1)w_1(\rho)$ than equation (\ref{YAE4-1}).
The main structure of the linearized equation is same with equation (\ref{YAE4-1}).
Since the parameter $\eps\ll1$ and the coefficient $e^{-\tau}$, those terms on $\eps w_0$ and $\eps (e^{-\tau}-1)w_1(\rho)$ do not effect the whole Nash-Moser iteration scheme. Hence using the same proof of process in Proposition 4.1, we can obtain this result.

\end{proof}

\subsection{Proof of Theorem 1.1.}

Let $T$ be a positive constant and 
$$
u_T^{\pm}(t,r)=\pm(T-t)\sqrt{1-(\frac{r}{T-t})^2},\quad (t,r)\in(0,T)\times(0,\sigma(T-t)].
$$

By Proposition 4.2, we have constructed a solution of the radially symmetric membranes equation (\ref{E1-1}) with the initial data (\ref{E1-2}) as follows
\bel{YAE4-19}
\aligned
u(t,r)&=u_T^{\pm}(t,r)+(1-\kappa)u_T^{\pm}(t,r)+w^{\infty}(t,r),
\endaligned
\ee
where $w^{\infty}(t,r)\in\HH^2(\Omega_{T-t})$ and
$$
\aligned
w^{\infty}(t,r)&:=(T-t)\Big(w^{(0)}(\log{T\over T-t},{r\over T-t})+\sum_{m=1}^{\infty}h^{(m)}(\log{T\over T-t},{r\over T-t})\\
&\quad\quad\quad+\eps w_0({r\over T-t})-\eps ({T-t\over T}-1)w_1({r\over T-t})\Big).
\endaligned
$$
Thus it follows from (\ref{YAE4-19}) that
\bel{YAE4-20}
u(t,r)-u_T^{\pm}(t,r)=(1-\kappa)u_T^{\pm}(t,r)+w^{\infty}(t,r).
\ee
So for a sufficient small positive constant $\eps (<\sigma)$ depending on $\sigma$, by Proposition 4.2 and (\ref{E4-5}), we can choose two positive parameters $\kappa$ and $T$ satisfying
$$
\aligned
&T^*-\delta<T<T^*+\delta,\quad for~0<\delta\ll1,\\
&1-(T\sigma)^{-{1\over2}}\Big(1-({\sigma T\over T^*})^2\Big)^{{1\over2}}\Big[1+\sigma+T\Big(1-({\sigma T\over T^*})^2\Big)\Big]^{-1}\eps<\kappa<1
\endaligned
$$ 
such that
$$
\aligned
&\|u_0(x)-u_{T^*}^{\pm}(0,r)\|_{\HH^{1}(\Omega_T)}+\|u_1(x)-\del_tu_{T^*}^{\pm}(0,r)\|_{\LL^{2}(\Omega_T)}\\
&\lesssim(1-\kappa)\Big(\|u_{T^*}^{\pm}(0,r)\|_{\HH^1(\Omega_T)}+\|\del_tu_{T^*}^{\pm}(0,r)\|_{\LL^2(\Omega_T)}\Big)+\|w^{\infty}(t,r)\|_{\HH^1(\Omega_T)}+\|w_t^{\infty}(t,r)\|_{\LL^2(\Omega_T)}\\
&\leq(1-\kappa)(T\sigma)^{{1\over2}}(1-({\sigma T\over T^*})^2)^{-{1\over2}}\Big[1+\sigma+T\Big(1-({\sigma T\over T^*})^2\Big)\Big]+2\eps\\
&\lesssim\eps,
\endaligned
$$
then by (\ref{YAE4-20}), we obtain
$$
\aligned
\| u(t,r)-u_T^{\pm}(t,r)\|_{\HH^1(\Omega_{T-t})}
&\leq(1-\kappa)\|u_{T^*}^{\pm}(t,r)\|_{\HH^1(\Omega_{T-t})}
+\|w^{\infty}(t,r)\|_{\HH^1(\Omega_{T-t})}\\
&\lesssim \eps(T-t),
\endaligned
$$
where we impose the boundary condition $$w^{\infty}(t,r)|_{r\in\del\Omega_{T-t}}=w_r^{\infty}(t,r)|_{r\in\del\Omega_{T-t}}=0,$$ and $$\Omega_{T-t}:=\{r:r\in(0,\sigma(T-t)]\}.$$

Therefore, two lightlike self-similar solutions $u_T^{\pm}(t,r)$
are nonlinearly stable in $\{(t,r): (t,r)\in (0,T)\times(0,\delta(T-t)]\}$.
\\


\section{Appendix}\setcounter{equation}{0}

In the appendix, we give the details on the proof of Lemma 3.5. Firstly, we recall a result of the existence of difference equation, which first established by Birkhoff and Trjitzinsk.

\begin{proposition}(Birkhoff and Trjitzinsk \cite{Bir3})
The $k$th-order linear difference equation
\begin{equation*}
a_0(n)u_n+a_1(n)u_{n+1}+\ldots+a_k(n)u_{n+k}=0,~~~a_0\not\equiv0,~~a_k\not\equiv0,
\end{equation*}
with polynomial coefficients $a_i$ has precisely $k$ linearly independent formal solutions of the general form
\begin{equation}\label{BT}
u_n=e^{Q(n)}n^r\sum_{i=0}^{\infty}n^{\frac{-i}{p}}\sum_{j=0}^mC_{i,j}\ln^{j}n,
\end{equation}
where
\begin{equation*}
Q(n)=\mu n\ln n+\sum_{s=0}^p\nu_s n^{\frac{s}{p}},
\end{equation*}
and $C_{i,j}$ and $\nu_s$ are coefficients, $r,p\in\mathbb{N}$, $\mu p\in\mathbb{Z}$ and $m\in\mathbb{N}\cap\{0\}$.
\end{proposition}

It is hard to get an exact expansion of solution to (\ref{E3-6}) by (\ref{BT}). So we have to use other method to analyze the asmptotic behavior of solutions to (\ref{E3-6}).

Let
\begin{eqnarray*}
&&z_{n}^{(1)}=a_n,\\
&&z_{n}^{(2)}=a_{n+1}=z^{(1)}_{n+1},\\
&&z_{n}^{(3)}=a_{n+2}=z^{(2)}_{n+1},\\
&&z_{n}^{(4)}=a_{n+3}=z^{(3)}_{n+1},
\end{eqnarray*}
then by (\ref{E3-6}), we have
\begin{equation*}
z_{n+1}^{(4)}=(2-p_1(n))z^{(3)}_{n}-(1+p_2(n))z^{(1)}_{n},
\end{equation*}
where $z_n^{(i)}$ $(i=1,2,3,4)$ depends on $\nu$ and $\kappa$.

Furthermore, let $z_n=(z^{(1)}_n,z^{(2)}_n,z^{(3)}_n,z^{(4)}_n)$, we have
\bel{RRRRE3-7}
\aligned
z_{n+1}&=\left(
\begin{array}{cccc}
0&1&0&0\\
0&0&1&0\\
0&0&0&1\\
-1-p_2(n)&0&2-p_1(n)&0
\end{array}
\right)z_n\\
&=(\mathcal{D}+\mathcal{T}(n))z_n,
\endaligned
\ee
where
\begin{eqnarray*}
\mathcal{D}=\left(
\begin{array}{cccc}
0&1&0&0\\
0&0&1&0\\
0&0&0&1\\
-1&0&2&0
\end{array}
\right),~~
\mathcal{T}(n)=\left(
\begin{array}{cccc}
0&0&0&0\\
0&0&0&0\\
0&0&0&0\\
-p_2(n)&0&-p_1(n)&0
\end{array}
\right).
\end{eqnarray*}

Since the matrix $\mathcal{D}$ has eigenvalues $\lambda_1=\lambda_2=1$ and $\lambda_3=\lambda_4=-1$ (double), we should diagonalize the matrix $\mathcal{D}$.
Direct computation shows that $\lambda_1$ has an eigenvector $\xi_1=(1,1,1,1)$, and $\lambda_3$ has an eigenvector $\xi_3=(1,-1,1,-1)$. Note that $1$ and $-1$ are double eigenvalues of $\mathcal{D}$.
We have to set $(\lambda_1E-\mathcal{D})\xi_2=\xi_1$, i.e.
\begin{eqnarray*}
\left(
\begin{array}{cccc}
1&-1&0&0\\
0&1&-1&0\\
0&0&1&-1\\
1&0&-2&1
\end{array}
\right)
\left(
\begin{array}{c}
x_1\\x_2\\x_3\\x_4
\end{array}
\right)=
\left(
\begin{array}{c}
1\\1\\1\\1
\end{array}
\right),
\end{eqnarray*}
solving it, we have a new eigenvector $\xi_2=(0,1,2,3)$. Here $E$ is the identity matrix.

Similarly, set $(\lambda_3E-\mathcal{D})\xi_4=\xi_3$, i.e.
\begin{eqnarray*}
\left(
\begin{array}{cccc}
-1&-1&0&0\\
0&-1&-1&0\\
0&0&-1&-1\\
1&0&-2&-1
\end{array}
\right)
\left(
\begin{array}{c}
x_1\\x_2\\x_3\\x_4
\end{array}
\right)=
\left(
\begin{array}{c}
1\\-1\\1\\-1
\end{array}
\right),
\end{eqnarray*}
we get the last eigenvector $\xi_4=(0,1,-2,3)$.

Let
\begin{eqnarray*}
\mathcal{P}=\left(
\begin{array}{cccc}
1&0&1&0\\
1&1&-1&1\\
1&2&1&-2\\
1&3&-1&3
\end{array}
\right),~~
\mathcal{P}^{-1}=\left(
\begin{array}{cccc}
\frac{1}{2}&\frac{3}{4}&0&-\frac{1}{4}\\
-\frac{1}{4}&-\frac{1}{4}&\frac{1}{4}&\frac{1}{4}\\
\frac{1}{2}&-\frac{3}{4}&0&\frac{1}{4}\\
\frac{1}{4}&-\frac{1}{4}&-\frac{1}{4}&\frac{1}{4}
\end{array}
\right),
\end{eqnarray*}
then the matrix $\mathcal{D}$ is transformed into Jordan matrix
\begin{eqnarray}\label{RRRRE3-8}
\mathcal{P}^{-1}\mathcal{D}\mathcal{P}=J:=\left(
\begin{array}{cccc}
1&1&0&0\\
0&1&0&0\\
0&0&-1&1\\
0&0&0&-1
\end{array}
\right).
\end{eqnarray}

So let $z_n=\mathcal{P}y_n^{(1)}$, by (\ref{RRRRE3-7}), we get
\begin{eqnarray}\label{RRRRE3-12}
y_{n+1}^{(1)}=(J+\mathcal{P}^{-1}\mathcal{T}(n)\mathcal{P})y_n^{(1)},
\end{eqnarray}
where
\begin{eqnarray*}
\mathcal{P}^{-1}\mathcal{T}(n)\mathcal{P}=J:=\left(
\begin{array}{cccc}
\frac{p_1(n)+p_2(n)}{4}&\frac{p_1(n)}{2}&\frac{p_1(n)+p_2(n)}{4}&-\frac{p_1(n)}{2}\\
-\frac{p_1(n)+p_2(n)}{4}&-\frac{p_1(n)}{2}&-\frac{p_1(n)+p_2(n)}{4}&\frac{p_1(n)}{2}\\
-\frac{p_1(n)+p_2(n)}{4}&-\frac{p_1(n)}{2}&-\frac{p_1(n)+p_2(n)}{4}&\frac{p_1(n)}{2}\\
-\frac{p_1(n)+p_2(n)}{4}&-\frac{p_1(n)}{2}&-\frac{p_1(n)+p_2(n)}{4}&\frac{p_1(n)}{2}
\end{array}
\right).
\end{eqnarray*}
Now our task is to transform the Jordan matrix $J$ into a diagonal matrix with four different eigenvalues.
\begin{lemma}
There are two inverse matrices $\mathcal{M}_1(n)$ and $\mathcal{M}_2(n)$ depending on $n$ such that
\begin{equation}\label{RRRRE3-11}
\mathcal{M}_1(n)J\mathcal{M}_2(n)=\left(
\begin{array}{cccc}
1&0&0&0\\
0&1+\frac{1}{n}&0&0\\
0&0&-1&0\\
0&0&0&-(1+\frac{1}{n})
\end{array}
\right),
\end{equation}
where $n=1,2,3,\ldots$ and $J$ is the Jordan matrix in (\ref{RRRRE3-8}).
\end{lemma}
\begin{proof}
This proof is based on observation. Let
\begin{eqnarray*}
\mathcal{P}_1=\left(
\begin{array}{cccc}
1&2&0&0\\
1&1&0&0\\
0&0&1&0\\
0&0&0&1
\end{array}
\right),~~
\mathcal{P}_1^{-1}=\left(
\begin{array}{cccc}
-1&2&0&0\\
1&-1&0&0\\
0&0&1&0\\
0&0&0&1
\end{array}
\right),
\end{eqnarray*}
and
\begin{eqnarray*}
\mathcal{P}_2=\left(
\begin{array}{cccc}
1&0&0&0\\
0&1&0&0\\
0&0&-1&-2\\
0&0&1&1
\end{array}
\right),~~
\mathcal{P}_2^{-1}=\left(
\begin{array}{cccc}
1&0&0&0\\
0&1&0&0\\
0&0&1&2\\
0&0&-1&-1
\end{array}
\right),
\end{eqnarray*}
then we derive
\begin{eqnarray}\label{RRRRE3-9}
\mathcal{P}_2^{-1}\mathcal{P}_1^{-1}J\mathcal{P}_1\mathcal{P}_2=\left(
\begin{array}{cccc}
0&-1&0&0\\
1&2&0&0\\
0&0&0&1\\
0&0&-1&-2
\end{array}
\right).
\end{eqnarray}

To diagonalize above matrix with four different eigenvalues, we introdcue a matrix depending on $n$
\begin{eqnarray*}
\mathcal{P}_3(n)=\left(
\begin{array}{cccc}
1&1&0&0\\
-1&-(1+\frac{1}{n})&0&0\\
0&0&1&0\\
0&0&0&1
\end{array}
\right),
\end{eqnarray*}
then
\begin{eqnarray*}
\mathcal{P}_3^{-1}(n+1)=\left(
\begin{array}{cccc}
n+2&n+1&0&0\\
-(n+1)&-(n+1)&0&0\\
0&0&1&0\\
0&0&0&1
\end{array}
\right).
\end{eqnarray*}
By (\ref{RRRRE3-9}), direct computation shows that
\begin{eqnarray}\label{RRRRE3-10}
\mathcal{P}_3^{-1}(n+1)\mathcal{P}_2^{-1}\mathcal{P}_1^{-1}J\mathcal{P}_1\mathcal{P}_2\mathcal{P}_3(n)=\left(
\begin{array}{cccc}
1&0&0&0\\
0&1+\frac{1}{n}&0&0\\
0&0&0&1\\
0&0&-1&-2
\end{array}
\right).
\end{eqnarray}
Thus we introduce another matrix
\begin{eqnarray*}
\mathcal{P}_4(n)=\left(
\begin{array}{cccc}
1&0&0&0\\
0&1&0&0\\
0&0&1&1\\
0&0&-1&-(1+\frac{1}{n})
\end{array}
\right),
\end{eqnarray*}
then
\begin{eqnarray*}
\mathcal{P}_4^{-1}(n+1)=\left(
\begin{array}{cccc}
1&0&0&0\\
0&1&0&0\\
0&0&n+2&n+1\\
0&0&-(n+1)&-(n+1)
\end{array}
\right).
\end{eqnarray*}
which combining with (\ref{RRRRE3-10}) gives that
\begin{eqnarray*}
\mathcal{M}_1(n)J\mathcal{M}_2(n)=\left(
\begin{array}{cccc}
1&0&0&0\\
0&1+\frac{1}{n}&0&0\\
0&0&-1&0\\
0&0&0&-(1+\frac{1}{n})
\end{array}
\right),
\end{eqnarray*}
where
\begin{eqnarray*}
&&\mathcal{M}_1(n)=\mathcal{P}_4^{-1}(n+1)\mathcal{P}_3^{-1}(n+1)\mathcal{P}_2^{-1}\mathcal{P}_1^{-1},\\
&&\mathcal{M}_2(n)=\mathcal{P}_1\mathcal{P}_2\mathcal{P}_3(n)\mathcal{P}_4(n).
\end{eqnarray*}
\end{proof}

We now return to the system (\ref{RRRRE3-12}). Set
\begin{eqnarray*}
&&y_n^{(1)}=\mathcal{P}_1y_n^{(2)},\\
&&y_n^{(2)}=\mathcal{P}_2y_n^{(3)},
\end{eqnarray*}
we derive
\begin{eqnarray}\label{RRRRE3-13}
y_{n+1}^{(3)}=(\mathcal{P}_2^{-1}\mathcal{P}_1^{-1}J\mathcal{P}_1\mathcal{P}_2+\mathcal{P}_2^{-1}\mathcal{P}_1^{-1}\mathcal{P}^{-1}\mathcal{T}(n)\mathcal{P}\mathcal{P}_1\mathcal{P}_2)y_n^{(3)}.
\end{eqnarray}

Set
\begin{eqnarray*}
&&y_n^{(3)}=\mathcal{P}_3(n)y_n^{(4)},~~y_{n+1}^{(3)}=\mathcal{P}_3(n+1)y_{n+1}^{(4)}\\
&&y_n^{(4)}=\mathcal{P}_4y_n^{(5)},~~y_{n+1}^{(4)}=\mathcal{P}_4y_{n+1}^{(5)},
\end{eqnarray*}
and
\begin{eqnarray*}
&&\mathcal{M}_1(n)=\mathcal{P}_4^{-1}(n+1)\mathcal{P}_3^{-1}(n+1)\mathcal{P}_2^{-1}\mathcal{P}_1^{-1},\\
&&\mathcal{M}_2(n)=\mathcal{P}_1\mathcal{P}_2\mathcal{P}_3(n)\mathcal{P}_4(n).
\end{eqnarray*}

Taking advantage of the process of proof in Lemma 5.1, we derive from
(\ref{RRRRE3-13}) that
\begin{eqnarray}\label{RRRRE3-14}
y_{n+1}^{(5)}&=&(\mathcal{M}_1(n)J\mathcal{M}_2(n)+\mathcal{M}_1(n)\mathcal{P}^{-1}\mathcal{T}(n)\mathcal{P}\mathcal{M}_2(n))y_n^{(5)}\nonumber\\
&=&(\tilde{\mathcal{D}}+\tilde{\mathcal{T}}(n))y_n^{(5)},~~~~n\geq1,
\end{eqnarray}
where $\tilde{\mathcal{D}}:=\mathcal{M}_1(n)J\mathcal{M}_2(n)$ is a diagonal matrix defined in (\ref{RRRRE3-11}), $\tilde{\mathcal{T}}(n)$ is a off-diagonal matrix, which is
\begin{equation}\label{RRRRE3-20}
\begin{array}{lll}
&&\tilde{\mathcal{T}}(n)=\mathcal{M}_1(n)\mathcal{P}^{-1}\mathcal{T}(n)\mathcal{P}\mathcal{M}_2(n)\\
&=&\left(
\begin{array}{cccc}
\frac{(n+4)(2p_1+p_2)}{4}&\frac{p_1(n+\frac{16}{n}+8)+p_2(n+\frac{8}{n}+6)}{4}&\frac{p_1(5n+2)+2(n+1)p_2}{4}&\frac{p_1(5n^2-2n-16)+2(n^2-4)}{4n}\\
\frac{-(n+1)(5p_1+3p_2)}{4}&\frac{-p_1(n+1)(5+\frac{8}{n})-p_2(3+\frac{4}{n})}{4}&\frac{(n+1)(p_1+p_2)}{4}&\frac{p_1(n^2+5n+4)+p_2(n^2+3n+2)}{4n}\\
\frac{(n+4)(p_1+p_2)}{4}&\frac{(n+4)[(1+\frac{4}{n})p_1+(1+\frac{2}{n})p_2]}{4}&\frac{-(n+4)(p_1+p_2)}{4}&\frac{-(n+4)(5p_1+3p_2)}{4n}\\
\frac{-(n+1)(p_1+p_2)}{4}&\frac{-(n+1)[(1+\frac{4}{n})p_1+(1+\frac{2}{n})p_2]}{4}&\frac{(n+1)(p_1-3p_2)}{4}&\frac{(n+1)(5p_1-7p_2)}{4n}
\end{array}
\right),
\end{array}
\end{equation}
where $p_1:=p_1(n)$ and $p_2:=p_2(n)$ defined in (\ref{E3-15})-(\ref{E3-16}), respectively.

So it follows from (\ref{RRRRE3-14}) that
\begin{eqnarray*}
y^{(5)}_{n+1}=(\tilde{\mathcal{D}}(n)+\tilde{\mathcal{T}}(n))\cdot(\tilde{\mathcal{D}}(n-1)+\tilde{\mathcal{T}}(n-1))\cdot\ldots\cdot(\tilde{\mathcal{D}}(1)+\tilde{\mathcal{T}}(1))y_1^{(5)},
\end{eqnarray*}
which combining with (\ref{RRRRE3-20}) that $y^{(5)}_{n+1}$ has an unbounded solution depending on $n$. This is coincident with the result of Birkhoff-Trjitzinsk \cite{Bir3}.
Thus we complete the proof of Lemma 3.5.


\textbf{Acknowledgments.} 
The author expresses his sincere thanks to Prof. Gang. Tian, Prof. Zhifei. Zhang, Prof. Dexing. Kong and Prof. Baoping Liu for their many kind helps and suggestions,
Prof. R. Donninger for giving me some important suggestions, Prof. J. Hoppe for his pointing out two explicit solutions being lightlike, and his suggestion and informing me his interesting papers \cite{hop}. The author also expresses his sincere thanks to Dr. C.H. Wei for his suggestion on the relationship between the timelike extremal hypersurface equation and Chaplygin gas model \cite{WY}.
The author is supported by NSFC No 11771359, and the Fundamental Research Funds for the Central Universities (Grant No. 20720190070, No.201709000061).



\begin{thebibliography}{xx}


\bibitem{Alin}  {\small
S. Alinhac, Existence d'ondes de rar\'{e}faction pour des syst$\grave{e}$mes quasi-lin\'{e}aires hyperboliques multidimensionnels. 
Comm. Partial Differential Equations 14 (1989), no. 2, 173-230.


\bibitem{Ba}
B.M. Barbashov, V.V. Nesterenko, A.M. Chervyakov, General solutions of nonlinear equations in the geometric theory of the relativistic string. Comm. Math. Phys. 84 (1982) 471-481.






\bibitem{Bir3}
D.G. Birkhoff, W.J. Trjitzinsky, Analytic theory of singular difference equations. Acta. Math. 60 (1933), 1-89.






\bibitem{BB}
P. Bizo\'{n} and P. Biernt, Generic self-similar blowup for equivariant wave maps and Yang-Mills fields in higher dimensions. Comm. Math. Phys. 338 (2015) 1443-1450.

\bibitem{Brie}
E. Brieskorn and H. Kn\"{o}rrer, Plane algebraic curves.Translated from the German by John Stillwell. Birkh\"{a}user Verlag, Basel, 1986.


\bibitem{Bus}
V.I. Buslaev, S.F. Buslaeva, Poincar\'{e} theorem on difference equations. Math. Zametki 78 (2005) 943-947.

\bibitem{Cos0}
O. Costin, M. Huang, W. Schlag, On the spectral properties of $L_{\pm}$ in three dimensions. Nonlinearity 25 (2012) 125-164.

\bibitem{Cos1}
O. Costin, R. Donninger, X. Xia: A proof for the mode stability of a self-similar wave map. Nonlinearity 29 (2016) 2451-2473.

\bibitem{Cos2}
O. Costin, R. Donninger, I. Glogi\'{c}, M. Huang, On the stability of self-similar solutions to nonlinear wave equations. Comm. Math. Phys. 343 (2016) 299-310.

\bibitem{Cos3}
O. Costin, R. Donninger, I. Glogi\'{c}, Mode stability of self-similar wave maps in higher dimensions. Comm. Math. Phys. 351 (2017) 959-972.



\bibitem{D}
R. Donninger, On stable self-similar blowup for equivariant wave maps. Comm. Pure Appl. Math. 64 (2011) 1029-1164.



\bibitem{D2}
R. Donninger and B. Sch\"{o}rkhuber, Stable blowup for wave equations in odd space dimensions.
Ann. Inst. H. Poincar\'{e} Anal. Non Lin\'{e}aire. 34 (2017) 1075-1354.



\bibitem{hop}
J. Hoppe, Some classical solutions of relativistic membrane equations in 4-space-time dimensions. Phys.
Lett. B 329 (1994) 10-14 .

\bibitem{Hop0}
J. Hoppe, email communication. (2017)

\bibitem{Egg}
J.Eggers, M.A. Fontelos, The role of self-similarity in singularities of partial
differential equations. Nonlinearity. 22 (2009) R1-R44

\bibitem{Hop1}
J. Eggers, J. Hoppe, Singularity formation for timelike extremal hypersurfaces. Physics Letters B. 680 (2009) 274-278.


\bibitem{Hop2}
J. Eggers, J. Hoppe, M. Hynek, N. Suramlishvili, Singularities of relativistic membranes. Geom. Flows. 1 (2015) 17-33.



\bibitem{Ela}
S. Elaydi, An introduction to difference equations. Undergraduate Texts in Mathematics, 3rd edn. Springer, New York (2005)

\bibitem{Ger}
T. Gerald, Ordinary differential equations and dynamical systems. American Mathematical Society, 2012.






\bibitem{H}
L. H\"{o}rmander,  Implicit function theorems. Stanford Lecture notes, University, Stanford 1977



\bibitem{H1}
L. H\"{o}rmander, The boundary problems of physical geodesy. Arch. Rational Mech. Anal. 62 (1976) 1-52.






\bibitem{Kong}
D.X. Kong, Q. Zhang and Q. Zhou, The Dynamics of Relativistic Strings Moving in the Minkowski Space. Comm. Math. Phys. 269 (2007) 153-174.









\bibitem{Kl}
K.J Engel, R. Nagel, One-Parameter Semigroups for Linear Evolution Equations, Graduate Texts in Mathematics, vol.194, Springer-Verlag, New York, 2000. With contributions by S. Brendle, M. Campiti, T. Hahn, G. Metafune, G. Nickel, D. Pallara, C. Perazzoli, A. Rhandi, S. Romanelli and R. Schnaubelt.


\bibitem{Lia}
J.F. Liang, A singular initial value problem and self-similar solutions
of a nonlinear dissipative wave equation. J. Diff. Eqn. 246 (2009) 819-844.




\bibitem{Lin}
H. Lindblad, A remark on global existence for small initial data of the minimal surface equation in Minkowskian space time. Proc. Amer. Math. Soc. 132 (2004) 1095-1102.


\bibitem{Moser}
J. Moser,  A rapidly converging iteration method and nonlinear partial differential equations I-II. \textit{Ann. Scuola Norm. Sup. Pisa.} \textbf{20}, (1966) 265-313, 499-535.


\bibitem{Mi}
T. Milnor, Entire timelike minimal surfaces in $E^{3,1}$. Michigan Math. J. 37 (1990)163-177.



\bibitem{Nash}
J. Nash, The embedding for Riemannian manifolds. \textit{Amer. Math.} \textbf{63}, (1956) 20-63.



\bibitem{Pa}
A. Pazy, Semigroups of Linear Operators and Applications to Partial Differential Equations, Springer-Verlag, New York, 1983.








\bibitem{tian}
L. Nguyen, G. Tian, On smoothness of timelike maximal cylinders in three-dimensional vacuum spacetimes. Classical Quantum Gravity. 30 (2013), no. 16, 165010, 26 pp.
















\bibitem{Sog}
C.D. Sogge, 
{\sl Lectures on Nonlinear Wave Equations,}
Monographs in Analysis, vol. II, International Press, Boston.







\bibitem{Wa}
H.S. Wall, Polynomials whose zeros have negative real parts. Amer. Math. Mon. 52 (1945) 308-322.


\bibitem{WY}
C.H. Wei, W.P. Yan, On the explicit self-similar motion of the relativistic Chaplygin gas.
Europhys. Lett. 122 (2018) 10005.


\bibitem{W}
E. Witten, Singularities in string theory. ICM 2002. Vol I. 495-504.


\bibitem{Yan}
W.P. Yan, The motion of closed hypersurfaces in the central force field. J. Diff. Eqns. 261 (2016), 1973-2005.



\bibitem{Yan1}
W.P. Yan, Dynamical behavior near explicit self-similar blow up solutions 
for the Born-Infeld equation. Nonlinearity. 32 (2019) 4682-4712.




\bibitem{YZ}
W.P. Yan, B.L. Zhang, Long time existence of solution for the bosonic membrane in the light cone gauge, J. Geometric. Anal. doi: 10.1007/s12220-019-00269-1.




\bibitem{ZY}
X. Zhao, W.P. Yan, Existence of standing waves for quasi-linear Schr\"{o}dinger equations on $\TT^n$. Adv. Nonlinear Anal. 9 (2020) 978-993.

}
























\end{thebibliography}
\end{document}